\documentclass[sn-mathphys-ay]{sn-jnl}
\usepackage[T1]{fontenc}
\usepackage{amsmath,amssymb,amsthm,amsfonts}
\usepackage{mathrsfs}
\usepackage{xcolor}
\usepackage{textcomp}
\usepackage{mathtools}
\usepackage[expansion=false]{microtype}
\usepackage[shortlabels]{enumitem}
\usepackage{bm}
\usepackage[nameinlink,capitalise]{cleveref}

\newcommand\xqed[1]{%
  \leavevmode\unskip\penalty9999 \hbox{}\nobreak\hfill
  \quad\hbox{{#1}}}
\newcommand\qeddef{\xqed{${\scriptstyle\spadesuit}$}}
\newcommand\qedthm{\xqed{${\scriptstyle\blacksquare}$}}

\setcitestyle{round}
\allowdisplaybreaks
\numberwithin{equation}{section}

\definecolor{linkcolor}{RGB}{87,41,40}
\hypersetup{
  colorlinks=true,
  linkcolor=linkcolor,
  citecolor=linkcolor,
  urlcolor=linkcolor
}

\theoremstyle{definition}
\newtheorem{theorem}{Theorem}[section]
\newtheorem{proposition}[theorem]{Proposition}
\newtheorem{lemma}[theorem]{Lemma}
\newtheorem{corollary}[theorem]{Corollary}
\newtheorem{conj}[theorem]{Conjecture}

\theoremstyle{definition}
\newtheorem{definition}[theorem]{Definition}

\theoremstyle{remark}
\newtheorem{remark}[theorem]{Remark}

\crefname{theorem}{Theorem}{Theorems}
\Crefname{theorem}{Theorem}{Theorems}
\crefname{proposition}{Proposition}{Propositions}
\Crefname{proposition}{Proposition}{Propositions}
\crefname{lemma}{Lemma}{Lemmas}
\Crefname{lemma}{Lemma}{Lemmas}
\crefname{corollary}{Corollary}{Corollaries}
\Crefname{corollary}{Corollary}{Corollaries}
\crefname{conj}{Conjecture}{Conjectures}
\Crefname{conj}{Conjecture}{Conjectures}
\crefname{definition}{Definition}{Definitions}
\Crefname{definition}{Definition}{Definitions}
\crefname{example}{Example}{Examples}
\Crefname{example}{Example}{Examples}
\crefname{remark}{Remark}{Remarks}
\Crefname{remark}{Remark}{Remarks}
\crefname{section}{Section}{Sections}
\Crefname{section}{Section}{Sections}

\newcommand{\term}[1]{\textbf{#1}}
\newcommand{\N}{\mathbb{N}}
\newcommand{\R}{\mathbb{R}}
\newcommand{\Borel}{\mathcal{B}}
\newcommand{\Clop}{\mathcal{C}}
\newcommand{\Ucal}{\mathcal{U}}
\newcommand{\Vcal}{\mathcal{V}}
\newcommand{\Wcal}{\mathcal{W}}
\newcommand{\Fcal}{\mathcal{F}}
\newcommand{\Acal}{\mathcal{A}}
\newcommand{\ind}{\mathbf{1}}
\newcommand{\symdiff}{\mathbin{\triangle}}
\newcommand{\given}{\,|\,}
\newcommand{\eps}{\varepsilon}
\newcommand{\set}[1]{\left\{#1\right\}}
\newcommand{\abs}[1]{\left|#1\right|}

\newcommand{\restrict}{\mathord{\upharpoonright}}

\usepackage{titlesec}

\titleformat{\section}
  {\normalfont\Large\bfseries\raggedright} 
  {\thesection.}               
  {1em}                        
  {}                           
\titleformat{\subsection}
  {\normalfont\large\bfseries\raggedright}
  {\thesubsection.}
  {1em}
  {}
\begin{document}

\title[The Truth Is the Hardest Thing]{Almost-Uniform Bayesian Convergence to the Truth Is Not Characterized by Countable Additivity on Conditional Hitting Times}

\author[1]{\fnm{M. Ali} \sur{Khan}}\email{akhan@jhu.edu}

\author*[2]{\fnm{Arthur Paul} \sur{Pedersen}}\email{appedersen@ccny.cuny.edu}

\author[3]{\fnm{Maxwell B.} \sur{Stinchcombe}}\email{max.stinchcombe@gmail.com}

\affil[1]{\orgdiv{Department of Economics}, \orgname{The Johns Hopkins University}}

\affil*[2]{\orgdiv{Department of Computer Science}, \orgname{The City College of New York \& The Graduate Center, The City University of New York}}

\affil[3]{\orgdiv{Department of Economics}, \orgname{The University of Texas at Austin}}

\abstract{We revisit the analysis of Bayesian convergence to the truth under
finite additivity in a recent paper by \citet{nielsen2021convergence}. Its
principal theorem proves that the posteriors of a probability function converge
to the truth almost uniformly if and only if the function has two properties: an
\textit{approximation property}, and \textit{countable additivity on conditional
hitting times}. We show that the two properties are necessary but not
sufficient, so that the theorem is false, and we locate the error in its
published proof. We construct a merely finitely additive probability function
that has both properties and whose posteriors converge to the truth almost
surely but not almost uniformly. Three of the paper's four remaining theorems,
and its corollary, lose their published proofs as well, two with the failed
implication and two to a separate defect that we also identify. Two of the four
results we reprove and one we leave undecided; the last is the corollary, which
our counterexample does not refute, and which we establish for a family
including the counterexample and leave open in general. We also show that
almost-sure convergence of posteriors to the truth for every event does not
characterize countable additivity. We close with what a repaired
characterization of almost-uniform convergence would have to add.}



\maketitle

\section{Introduction}

Bayesian convergence-to-the-truth theorems are usually proved under countable
additivity, commonly through martingale convergence
\citep{levy1937theorie,billingsley1995prob}.  Whether countable additivity should
be imposed on subjective probability is a foundational question,
and finitely additive probability has a long history in decision theory and the
foundations of probability
\citep{deFinettiArtofGuessing,rao1983charges,purves_sudderth_1976}.

What hangs on the question is the reach of an epistemological argument.
The practice of appealing to convergence-to-the-truth theorems has become  a canonical  instrument of philosophical argument. In epistemology,  they have been used to argue that Bayesianism has more to offer our troubled minds than internal consistency  at any given time:  The discipline imposed by Bayesian learning   on changing states of mind is assured to track the truth
in the long run.  In the philosophy of science, the same theorems
have been enlisted on the realist side of the debate with the skeptic over
whether inquiry can home in on the truth.   Both appeals are contested, and contested
on the ground that they turn on countable additivity.
\citet[p.~181]{juhl1994realism} argue that countable additivity, far from being
the ``technical convenience'' it is sometimes billed as, is ``a powerful axiom
for scientific realism that should be subjected to the greatest possible
philosophical scrutiny''; and the dispute over whether the convergence theorems
commit Bayesians to an immodest certainty turns on the same axiom
\citep{belot2013bayesian,elga2016bayesian,nielsen2019obligation}.

We adjudicate
none of that.  We observe only that its terms depend on which convergence
results hold and under what assumptions: a guarantee that requires countable
additivity is unavailable where degrees of belief are merely finitely additive
--- finitely but not countably additive.  So the question of which modes of
convergence survive the weakening is the question of which learning guarantees
remain, and in what form.  That is the question
\citet{nielsen2021convergence} sets out to answer.

Among the modes of convergence he studies is \textit{almost-uniform
convergence}, which is stronger than almost-sure convergence.  Almost-sure
convergence requires the posterior probability of a hypothesis to approach its
truth value outside an event of probability zero.  Almost-uniform convergence
requires it to do so uniformly outside an event of arbitrarily small
probability.

\citet{nielsen2021convergence} claims that two additional criteria, the
\textit{approximation property} and \textit{countable additivity on conditional
hitting times} (CHT), are individually necessary and jointly sufficient for
almost-uniform convergence to the truth:
\begin{align}
  \text{almost-uniform convergence}
    \quad&\Longleftrightarrow\quad
    \text{approximation and CHT}.
    \label{eq:characterization}
\end{align}
Read from left to right, \eqref{eq:characterization} says that the two criteria
are \emph{necessary}; read from right to left, that they are \emph{sufficient}.
The approximation property requires every event to be approximated by an event
whose truth is decided after finitely many observations; CHT requires the first
time at which a posterior crosses a fixed threshold to have a countably
additive distribution.  The characterization is his
Theorem~2 \citep[p.~406]{nielsen2021convergence}.  We show that it is false.

The necessity direction of Theorem~2 is correct.  The sufficiency direction is
false, and its published proof
\citep[pp.~406--407]{nielsen2021convergence} is invalid: it selects the
exceptional event after fixing both a probability budget and an accuracy
threshold, where almost-uniform convergence requires a single event, fixed
first, that serves every threshold at once.  Under countable additivity the
mismatch is harmless, since one may assign summable budgets to a sequence of
thresholds and take a countable union.  Under finite additivity there is no
bound on the probability of that union.

The mismatch is substantive.  We construct a probability function $P$ that:
\begin{enumerate}[itemsep=.4em,leftmargin=3em,labelsep=1em,topsep=1em]
\item[$\diamond$] gives every nonempty cylinder strictly positive probability;
\item[$\diamond$]  approximates every Borel event by a clopen event in $P$-symmetric
      difference;
\item[$\diamond$] satisfies CHT for every event and every threshold; and
\item[$\diamond$]  has posteriors that converge to the truth almost surely but not almost
      uniformly.
\end{enumerate}

The probability function is a mixture of two parts: a two-valued finitely
additive probability, which reads the truth value of an event along a sequence
of paths converging to a limit path, and a countably additive measure that
places an atom on each of those paths.

The counterexample defeats diagonalization.  Its posterior errors can be driven
below any fixed threshold outside a small exceptional event, but no single small
event works for every threshold at once; the countable union that would assemble
such an event is the diagonalization, and finite additivity does not license it.
The construction singles out a sequence of paths converging to a limit path; at
the $k$th path, the posterior error after
$k$ observations equals a number $\eta_k$.  For every fixed threshold $r>0$, the
inequality $\eta_k\le r$ holds on an ultrafilter-large set of indices, and that
is all CHT requires.  No single ultrafilter-large set of indices has
$\eta_k\to0$ along it, and that is what almost-uniform convergence would
require.

Finite additivity is not itself the obstacle.  By \cref{thm:positive}, replacing
the sequence $(\eta_k)$ by a sequence converging to zero in the
ordinary sense restores almost-uniform convergence, and the resulting
probability function is still merely finitely additive.  What fails is not
finite additivity but the right-to-left direction of \eqref{eq:characterization}.

Four of Nielsen's other claims are affected.  The first is
Theorem~3 \citep[p.~409]{nielsen2021convergence}, which asserts that there are
uncountably many merely finitely additive probabilities whose posteriors
converge almost uniformly.  Its proof fixes the values of a regular Borel
probability on every open set and then seeks nontrivial finitely additive Borel
extensions.  No such extensions exist: agreement on all open sets forces
agreement on every Borel set.  The statement of Theorem~3 is nonetheless true,
and we prove it.

The second is Theorem~4 \citep[p.~410]{nielsen2021convergence}, whose proof
invokes the same extension claim and fails with it.  We do not decide the
standalone truth of that theorem.

The third is Corollary~1 \citep[p.~408]{nielsen2021convergence}, a loss made
good only in part.  It draws almost-sure convergence to the truth from the
approximation property and CHT.  Its only
published derivation runs through the sufficiency direction of Theorem~2, so it
is left without a proof.  Our counterexample does not refute it --- the
probability function we construct does converge almost surely --- and for an
arbitrary probability function we record its claim as open.  For the family to
which our counterexample belongs we settle it affirmatively: there CHT alone
implies almost-sure convergence to the truth, so the almost-sure convergence of
our $P$ is forced rather than incidental.

The fourth is the second claim of Theorem~5
\citep[p.~411]{nielsen2021convergence}, a loss repaired.  That claim --- that
the probabilities the theorem constructs, which have the approximation property
but do not converge to the truth almost surely, therefore lack CHT --- is
inferred there from the sufficiency direction of Theorem~2, and so is left
without a proof as well.  We reprove it directly.  Because that inference is
also what supports Nielsen's claim that the approximation property does not
imply CHT, one direction of the
independence of his two criteria is restored with it; the other direction rests
on Theorem~4.

The presentation of technical material is self-contained.  We define every
finitely additive and ultrafilter notion the construction uses, and invoke three
standard results from countably additive probability: L\'evy's upward theorem,
Egorov's theorem, and regularity of Borel probabilities on compact metric
spaces.  The first two we quote from \citet{billingsley1995prob}; the third we
prove for Cantor space in \cref{lem:regularity}.

This paper is organized as follows.  \cref{sec:background} fixes notation,
defines the properties at issue, and states the results in question.
\cref{sec:necessity} proves the necessity
direction of \eqref{eq:characterization} and assembles the countably additive
facts used later.  The refutation then occupies three sections:
\cref{sec:gap} exhibits the published proof of sufficiency and locates its
defect, \cref{sec:construction} constructs a probability function that realizes
the defect --- beginning with the requirements that dictate its shape --- and
\cref{sec:main} verifies that both criteria hold while almost-uniform
convergence fails.  The remaining sections repair and extend.
\cref{sec:existence} shows that the extension argument behind Theorem~3 cannot
work and proves the statement it was meant to establish.  \cref{sec:as-ca}
shows that almost-sure convergence to the truth does not characterize countable
additivity, settling the first natural conjecture about a problem Nielsen
leaves open.  \cref{sec:diagonal} states the diagonal condition a repair of
Theorem~2 would require, and \cref{sec:conclusion} takes stock of what the
paper establishes and what it leaves open.

\section{Finitely Additive Posteriors on Cantor Space}\label{sec:background}

This section fixes the objects the rest of the paper works with: the sample
space and its finite-information algebras, the finitely additive probability
functions defined on them and the posteriors they generate, the two modes of
convergence to the truth that are being compared, and the two criteria
\citet{nielsen2021convergence} advances as characterizing the stronger mode.
It closes by stating, in this notation, the results of his that are at issue.

\subsection{Cantor Space and Finite Information}

Throughout, complements are taken in the ambient space: $A^c=\Omega\setminus A$
for $A\subseteq\Omega$, and $X^c=\N\setminus X$ for $X\subseteq\N$.  For a set
$A$, its \term{indicator function} $\ind_A$ is given by:
\[
  \ind_A(\omega)\quad=\quad
  \begin{cases}
    1&\mbox{if }\omega\in A;\\
    0&\mbox{if }\omega\notin A,
  \end{cases}
\]
and the \term{symmetric difference} of two events is:
\[
  A\symdiff B\quad=\quad(A\setminus B)\cup(B\setminus A).
\]
For a set $S$, we write $2^{S}$ for the collection of all subsets of $S$, and we
call a set $X\subseteq S$ \term{cofinite in $S$} if $S\setminus X$ is finite.

Let:
\[
  \Omega=\bigl\{0,1\bigr\}^{\N}, \qquad \N=\{1,2,3,\ldots\},
\]
equipped with the product topology, where $\{0,1\}$ itself is equipped with the
discrete topology.  An element $\omega\in\Omega$ is a \term{path}, written
coordinatewise as $\omega=(\omega_1,\omega_2,\ldots)$ with $\omega_n\in\{0,1\}$;
the coordinate $\omega_n$ is the $n$th observation along $\omega$.

A \term{word} is a finite string $s=s_1\cdots s_n$ with $s_i\in\{0,1\}$, and its
\term{length} is $\abs{s}=n$.  We allow $n=0$, in which case $s$ is the
\term{empty word}.  The \term{cylinder} determined by $s$ is:
\[
  [\,s\,]\quad=\quad
  \bigl\{\,\omega\in\Omega\,:\,\omega_1=s_1,\,\ldots,\,\omega_n=s_n\,\bigr\},
\]
so that $[\,s\,]=\Omega$ when $s$ is empty, and we call $[\,s\,]$ a
\term{length-$n$ cylinder} when $\abs{s}=n$.  Words and paths are combined by
\term{concatenation}, written as juxtaposition: for words $s$ and $t$, the word
$st$ consists of the coordinates of $s$ followed by those of $t$; and for a word
$s$ and a path $x$, the path $sx$ is given by:
\[
  sx\quad=\quad\bigl(\,s_1,\ldots,s_{\abs{s}},\,x_1,x_2,\ldots\,\bigr).
\]
For $m\ge0$ we write $0^m$ for the word consisting of $m$ zeros, and $0^\infty$
for the path all of whose coordinates are $0$.  If $\omega\in\Omega$, then
$\omega\restrict n=\omega_1\cdots\omega_n$ and $[\omega\restrict n]$ is the
unique length-$n$ cylinder containing $\omega$.

The cylinders form a countable basis for the product topology, and each of them
is \term{clopen}, that is, both open and closed.  Being a product of finite
discrete spaces, $\Omega$ is compact; it is
also completely metrizable.  Let $\Clop$ denote the algebra of all clopen subsets of $\Omega$, let
$\Borel=\sigma(\Clop)$ denote the Borel $\sigma$-algebra generated by $\Clop$, and let $\mathscr F_n$ denote
the finite $\sigma$-algebra generated by the length-$n$ cylinders.  Thus
$\mathscr F_n\subseteq\mathscr F_{n+1}$ for all $n\in \N$ and
$\sigma(\bigcup_n\mathscr F_n)=\Borel$.

The algebra $\mathscr F_n$ is the algebra of events settled by the first $n$
observations.  We record its elementary structure, since the construction below
repeatedly needs to know that certain events are clopen.

\begin{lemma}[Finite information]\label{lem:finite-info}
Let $n\in\N$.  A subset of $\Omega$ belongs to $\mathscr F_n$ if and only if it
is a union of length-$n$ cylinders.  Every such set is clopen, and
$\mathscr F_n$ is finite.  Consequently, if $g:\Omega\to\R$ is constant on each
length-$n$ cylinder, then $\{\omega:g(\omega)>c\}$ and
$\{\omega:g(\omega)\ge c\}$ are clopen for every $c\in\R$.\qedthm
\end{lemma}

\begin{proof}
The $2^n$ length-$n$ cylinders are pairwise disjoint and cover $\Omega$.  The
unions of subfamilies of a finite partition form a $\sigma$-algebra, here with
$2^{2^n}$ members, and this $\sigma$-algebra contains the length-$n$ cylinders
and is contained in every $\sigma$-algebra that contains them; it is therefore
$\mathscr F_n$.  A finite union of clopen sets is clopen.  If $g$ is constant on
each length-$n$ cylinder, then $\{\omega:g(\omega)>c\}$ is the union of those
length-$n$ cylinders on which the constant value of $g$ exceeds $c$, and
similarly with $\ge$ in place of $>$.
\end{proof}

\subsection{Probability Functions and Their Posteriors}

A \term{probability function} on $\Borel$ is a map
$P:\Borel\to[0,1]$ with $P(\Omega)=1$ and
\[
  P(A\cup B)=P(A)+P(B)
\]
whenever $A$ and $B$ are disjoint.  Thus a probability function is normalized,
nonnegative, and finitely additive.

A \term{probability measure} is a probability function that is countably additive:
for every pairwise disjoint sequence $(A_j)_{j=1}^{\infty}$ of events in $\Borel$:
\[
  P\left(\bigcup_{j=1}^{\infty}A_j\right)
  \quad=\quad\sum_{j=1}^{\infty}P(A_j).
\]
A probability function that is not countably additive is \term{merely finitely
additive}.  We use this terminology to keep the difference between finite and
countable additivity visible; in the language of
\citet[Definition~2.1.1(7), p.~35]{rao1983charges}, a probability function is a
probability charge.

For $x\in\Omega$, \term{point mass} at $x$ is the probability measure $\delta_x$
given by:
\[
  \delta_x(A)\quad=\quad\ind_A(x)
  \qquad (A\in\Borel).
\]
A path $x$ is an \term{atom} of a probability measure $\mu$ on $\Borel$ if
$\mu(\{x\})>0$, and $\mu$ is \term{diffuse} if it has no atoms.  Every countable
subset of $\Omega$ is Borel, since singletons are closed, and countable
additivity therefore makes diffuseness equivalent to the condition that
$\mu(A)=0$ for every countable $A\subseteq\Omega$.  Finally, the
\term{fair-coin measure} $Q$ on $\Omega$ is the probability measure that gives
each cylinder the probability of the corresponding run of tosses of a fair coin:
\begin{equation}\label{eq:faircoin}
  Q\bigl([\,s\,]\bigr)\quad=\quad2^{-\abs{s}}
  \qquad\mbox{for every word }s.
\end{equation}
Its existence is the standard construction of a countable product of copies of
the uniform probability on $\{0,1\}$;
\citet[Section~2, pp.~24--27]{billingsley1995prob} obtains it without
product-measure machinery, by transferring Lebesgue measure from $[0,1)$ to
$\Omega$ along dyadic expansions, the existence of the transferred measure
resting on his Theorem~2.2.  Nothing below
uses more about $Q$ than \eqref{eq:faircoin} and countable additivity; in
particular, $Q$ is diffuse, because $Q(\{\omega\})\le Q([\omega\restrict n])=2^{-n}$
for every $n$.

\begin{remark}[Nielsen's terminology and notation]\label{rem:terminology}
The conventions of \citet{nielsen2021convergence} differ from ours and should be
translated when the two papers are read side by side.  There, ``probability
measure'' means what we call a probability function; ``countably additive
probability measure'' is the strengthened notion we call a probability measure;
and ``merely finitely additive'' agrees with our usage.  Nielsen writes
$\omega^n$ for the cylinder we write $[\omega\restrict n]$, $\mathcal C$ for our
clopen algebra $\Clop$, $\mathcal F$ for our Borel $\sigma$-algebra $\Borel$, and
$\mathcal C_n$ for our finite $\sigma$-algebra $\mathscr F_n$.  He calls
\eqref{eq:posi} below Assumption~1.\qedthm
\end{remark}

We impose Nielsen's positivity assumption:
\begin{equation}\label{eq:posi}
  P\Bigl(\bigl[\omega\restrict n\bigr]\Bigr)>0
  \qquad \mbox{for all }\omega\in\Omega\mbox{ and }n\in\N.
\end{equation}
This assumption ensures that elementary conditional probabilities are always
defined.  For $A\in\Borel$, the \term{time-$n$ posterior} of $A$ --- the
posterior probability of $A$ after the first $n$ observations --- is the
function:
\begin{equation}\label{eq:posterior}
  P_n(A)(\omega)
  \quad=\quad P\Bigl(A\Bigm|\bigl[\omega\restrict n\bigr]\Bigr)
  \quad=\quad\frac{P\Bigl(A\cap\bigl[\omega\restrict n\bigr]\Bigr)}{P\Bigl(\bigl[\omega\restrict n\bigr]\Bigr)}.
\end{equation}
The function $P_n(A)$ is constant on each length-$n$ cylinder, so by
\cref{lem:finite-info} it is $\mathscr F_n$-measurable and every event of the
form $\{P_n(A)>c\}$ or $\{P_n(A)\ge c\}$ is clopen.

\subsection{Modes of Convergence to the Truth}

Almost-uniform convergence strengthens almost-sure convergence.  In place of
pointwise convergence on an event of probability one, it demands uniform
convergence off an event of arbitrarily small probability --- and it demands
that the event be chosen before the accuracy the uniformity must achieve.

\begin{definition}[Almost-uniform convergence]\label{def:au-general}
Let $P$ be a probability function on $\Borel$, and let $f$ and
$f_1,f_2,\ldots$ be real-valued Borel functions on $\Omega$.  Then
$(f_n)_{n=1}^{\infty}$ \term{converges to $f$ almost uniformly} if for every
$\eps>0$ there is an event $E\in\Borel$ with $P(E)<\eps$ such that $f_n\to f$
uniformly on $E^c$: for every $r>0$ there is $N\in\N$ with:
\[
  n\ge N,\ \omega\notin E
  \quad\Longrightarrow\quad
  \bigl|f_n(\omega)-f(\omega)\bigr|<r.
\]
The event $E$ may depend on $\eps$ but not on $r$.\qeddef
\end{definition}

This is \citet[Definition~3, p.~402]{nielsen2021convergence}.  Applied to the
posterior processes of a probability function, it yields the two learning
properties at issue.

\begin{definition}[Almost-sure and almost-uniform learning]\label{def:convergence}
Let $P$ be a probability function on $\Borel$, and let $A\in\Borel$.
\begin{enumerate}[(a)]
\item The posteriors  are said to converge to the truth \term{almost surely} for $A$ if
\[
  P\Biggl(\Bigl\{\,\omega\in \Omega\,:\,P_n(A)(\omega)\,\to\,\ind_A(\omega)\,\Bigr\}\Biggr)\quad=\quad1.
\]
\item The posteriors are said to converge to the truth \term{almost uniformly} for $A$ if
$P_n(A)\to\ind_A$ almost uniformly in the sense of \cref{def:au-general}: for every
$\eps>0$, there is $E\in\Borel$ with $P(E)<\eps$ such that, for every $r>0$,
there is $N$ with:
\[
  n\ge N,\ \omega\notin E
  \quad\Longrightarrow\quad
  \Bigl|\,P_n(A)(\omega)\,-\,\ind_A(\omega)\,\Bigr|\;<\;r.
\]
\end{enumerate}
The probability function $P$ has either convergence property \term{to the truth}
(simpliciter) if it has that property for every $A\in\Borel$.\qeddef
\end{definition}

Almost-uniform convergence implies almost-sure convergence, and the implication
uses only finite additivity \citep[Lemma~3, p.~403]{nielsen2021convergence}: if
$E_m$ satisfies $P(E_m)<1/m$ with uniform convergence off $E_m$, then
$\bigcup_mE_m^c$ has probability one by monotonicity, and the convergence holds
at each of its points.  The converse implication holds under countable
additivity, by Egorov's theorem, but not in general.  Its failure is the subject
of this paper.  The order
of the quantifiers in part~(b) is what does the work below: the exceptional event $E$ may
depend on $A$ and $\eps$, but it must not depend on the accuracy level $r$.

\subsection{Approximation and Conditional Hitting Times}

\begin{definition}[Approximation]\label{def:approx}
A probability function $P$ on $\Borel$ has the \term{approximation property} if for every
$A\in\Borel$ and every $\eps>0$, there is a clopen $C\in\Clop$ such that:
\begin{equation*}
  P(A\symdiff C)<\eps.
  \tag*{\qeddef}
\end{equation*}
\end{definition}

A clopen event is decided by finitely many coordinates.  The approximation
property therefore says that every Borel event is close, in $P$-probability, to an
event that can be verified or refuted after finitely many observations.

For $A\in\Borel$ and $\delta\in(0,1)$, let:
\begin{equation}\label{eq:Cn}
  C_n(A,\delta)\quad=\quad\Bigl\{\,\omega\in \Omega\,:\,P_n(A)(\omega)\,>\,\delta\,\Bigr\}.
\end{equation}
This is the event on which the time-$n$ posterior exceeds $\delta$.  Define the
first-hitting time:
\[
  \tau_{A,\delta}(\omega)
  \quad=\quad\inf\Bigl\{n\in\N:\omega\in C_n(A,\delta)\Bigr\},
\]
with $\inf\varnothing=\infty$, and put:
\begin{align}
  D_1(A,\delta)\quad&=\quad C_1(A,\delta),\nonumber\\
  D_n(A,\delta)
    \quad&=\quad C_n(A,\delta)\setminus\bigcup_{j<n}C_j(A,\delta),
      \qquad n\ge2.\label{eq:Dn}
\end{align}
Thus $D_n(A,\delta)=\bigl\{\tau_{A,\delta}=n\bigr\}$.  By
\cref{lem:finite-info}, each $C_n(A,\delta)$ belongs to $\mathscr F_n$ and is
therefore clopen; hence each $D_n(A,\delta)$ is clopen as well, and the $D_n$ are
pairwise disjoint.

\begin{definition}[Conditional hitting times]\label{def:cht}
A probability function $P$ on $\Borel$ is said to be \term{countably additive on
conditional hitting times}, abbreviated \term{CHT}, if:
\begin{equation}\label{eq:cht}
 P\left(\bigcup_{n=1}^{\infty}D_n(A,\delta)\right)
   =\sum_{n=1}^{\infty}P(D_n(A,\delta))
\end{equation}
for every $A\in\Borel$ and $\delta\in(0,1)$.\qeddef
\end{definition}

Thus CHT requires the distribution of each first-hitting time $\tau_{A,\delta}$
to be countably additive on its finite values, even though $P$ need not be
countably additive on arbitrary disjoint sequences.  Let:
\[
  O_N(A,\delta)\;=\;\bigcup_{n=1}^N C_n(A,\delta)\;\mbox{ for each }N\in\N
  \qquad \mbox{and}  \qquad
  O(A,\delta)\;=\;\bigcup_{n=1}^{\infty}C_n(A,\delta).
\]
The partial unions satisfy
$O_N(A,\delta)=\bigcup_{n=1}^ND_n(A,\delta)$.  Hence, by finite additivity, CHT
is equivalent to
\begin{equation}\label{eq:cht-continuity}
  P\Bigl(O(A,\delta)\Bigr)\quad=\quad\lim_{N\to\infty}P\Bigl(O_N(A,\delta)\Bigr).
\end{equation}
Thus CHT is continuity from below but only for this special family of increasing
clopen events.

All standing notation is now fixed, and we collect it for reference.  The sample
space is $\Omega=\{0,1\}^{\N}$; $[s]$ is the cylinder determined by a finite word
$s$, and $[\omega\restrict n]$ is the length-$n$ cylinder containing $\omega$;
$\Clop$ is the algebra of clopen events, $\Borel=\sigma(\Clop)$ the Borel
$\sigma$-algebra, and $\mathscr F_n$ the finite $\sigma$-algebra of events
settled by the first $n$ observations.  For a probability function $P$ satisfying
\eqref{eq:posi} and an event $A\in\Borel$, the function $P_n(A)$ is the time-$n$
posterior \eqref{eq:posterior}; $C_n(A,\delta)$ is the event that this posterior
exceeds $\delta$; $\tau_{A,\delta}$ is the first time at which it does so;
$D_n(A,\delta)=\{\tau_{A,\delta}=n\}$; and $O_N(A,\delta)$ and $O(A,\delta)$ are
the partial and total unions of the events $C_n(A,\delta)$.

\subsection{The Results at Issue}\label{sec:nielsen-results}

We state the results of \citet{nielsen2021convergence} that this paper
concerns, in the terminology fixed above.  Each is stated there under
Assumption~1, which is \eqref{eq:posi}.  Throughout the paper, a result cited by
a bare number --- Theorem~2, Corollary~1, and so on --- is one of Nielsen's;
our own results carry a section number.

\begin{enumerate}[leftmargin=6.2em,labelsep=1em,itemsep=.6em,topsep=.8em]
\item[\textit{Theorem~2}] \citep[p.~406]{nielsen2021convergence}
      A probability function on $\Borel$ converges to the truth almost uniformly
      if and only if it has the approximation property and the CHT property.
\item[\textit{Corollary~1}] \citep[p.~408]{nielsen2021convergence}
      A probability function on $\Borel$ with the approximation property and the
      CHT property converges to the truth almost surely.
\item[\textit{Theorem~3}] \citep[p.~409]{nielsen2021convergence}
      There are uncountably many merely finitely additive probability functions
      on $\Borel$ that converge to the truth almost uniformly.
\item[\textit{Theorem~4}] \citep[p.~410]{nielsen2021convergence}
      There are uncountably many probability functions on $\Borel$ that have the
      CHT property but not the approximation property.
\item[\textit{Theorem~5}] \citep[p.~411]{nielsen2021convergence}
      There are uncountably many probability functions on $\Borel$ that have the
      approximation property but do not converge to the truth almost surely, and
      that therefore do not have the CHT property.
\end{enumerate}

\cref{sec:necessity} proves the necessity half of Theorem~2 and
\cref{sec:main} refutes its sufficiency half.  Corollary~1 is derived in
\citet{nielsen2021convergence} from Theorem~2 together with the implication
from almost-uniform to almost-sure convergence
\citep[Lemma~3, p.~403]{nielsen2021convergence}; with the sufficiency half gone,
that derivation lapses, and \cref{rem:corollary} records what remains of the
corollary.  Theorem~3 and Theorem~4 are proved by a single extension argument,
which \cref{sec:published-proof} shows to be unavailable;
\cref{thm:positive} reproves Theorem~3 by other means.  Theorem~5 does not rest
on that extension argument, but the second of its two claims is inferred from
the sufficiency half of Theorem~2 and so lapses with it;
\cref{rem:nielsen-thm5} explains this and \cref{cor:nielsen-thm5} reproves the
claim.  Immediately after the proof of Theorem~3,
\citet[p.~410]{nielsen2021convergence} infers that there are uncountably many
merely finitely additive probability functions
converging to the truth almost surely; \cref{thm:positive} secures that
conclusion as well.  Finally, \citet[p.~408]{nielsen2021convergence} leaves the
characterization of almost-sure convergence to the truth open, and returns to it
among the questions for future research \citep[p.~412]{nielsen2021convergence};
\cref{sec:as-ca} eliminates the first candidate, countable additivity itself.

\section{Necessity of Approximation and CHT}\label{sec:necessity}

Nielsen's Theorem~2 claims, under \eqref{eq:posi}, that almost-uniform
convergence to the truth is equivalent to the conjunction of approximation and
CHT.  The necessity direction is sound.  We prove it here, since the
counterexample of \cref{sec:main} leaves it standing and since its proof
isolates the tail estimate that the rest of the paper turns on.

\begin{lemma}[Tail control from almost-uniform convergence]\label{lem:au-tail}
Let $P$ be a probability function on $\Borel$, and let $(f_n)^{\infty}_{n=1}$ and $f$ be
real-valued Borel functions on $\Omega$.  If $f_n\to f$ almost uniformly, then
for every $r>0$:
\begin{equation*}
  P\left(\left\{\omega:\sup_{m\ge n}\bigl|f_m(\omega)-f(\omega)\bigr|>r\right\}\right)
  \longrightarrow0
  \qquad\mbox{as }n\to\infty.
  \tag*{\qedthm}
\end{equation*}
\end{lemma}

\begin{proof}
Fix $r>0$ and an arbitrary $\gamma>0$.  Choose $E$ with $P(E)<\gamma$ such that
$f_n\to f$ uniformly on $E^c$.  For all sufficiently large $n$, the displayed
tail event is contained in $E$, and hence has probability less than $\gamma$.
Since $\gamma$ is arbitrary, the probabilities tend to zero.
\end{proof}

The lemma converts almost-uniform convergence into a tail estimate: for each
accuracy $r$, the probability that the error will ever again exceed $r$ after
time $n$ tends to zero as $n$ grows.  It is
\citet[Lemma~2, p.~402]{nielsen2021convergence}, stated with the accuracy
threshold $r$ separated from the probability budget rather than tied to it; the
proof is his.

\begin{proposition}[The valid half of Nielsen's characterization]
\label{prop:necessity}
Let $P$ be a probability function on $\Borel$.  If $P$ satisfies
\eqref{eq:posi} and converges to the truth almost uniformly, then $P$ has the
approximation property and the CHT property.
\qedthm \end{proposition}

\begin{proof}
Fix $A\in\Borel$.

\emph{Approximation.}
Since $|P_n(A)-\ind_A|\le\sup_{m\ge n}|P_m(A)-\ind_A|$ pointwise,
\cref{lem:au-tail} gives:
\[
 P\bigl(|P_n(A)-\ind_A|>r\bigr)\longrightarrow0
 \qquad\mbox{ for each } r>0.
\]
Define the event:
\[
  B_n\;=\;\Bigl\{\,\omega\in \Omega:P_n(A)(\omega)\,\ge\,\tfrac{1}{2}\,\Bigr\},
\]
which is clopen by \cref{lem:finite-info}.
If $\omega\in A\symdiff B_n$, then
$|P_n(A)(\omega)-\ind_A(\omega)|\ge\tfrac{1}{2}$.  Therefore:
\[
  P(A\symdiff B_n)
  \;\le\; P\Bigl(\,\Bigl|P_n(A)-\ind_A\Bigr|>\tfrac{1}{3}\,\Bigr)\longrightarrow0,
\]
which proves the approximation property.

\emph{Conditional hitting times.}
Fix $\delta\in(0,1)$, abbreviate $D_n=D_n(A,\delta)$, and set
$D=\bigcup_nD_n$.  For every $N\in\N$:
\begin{align*}
 P(A\cap D)
 \quad&=\quad\sum_{n=1}^N P(A\cap D_n)
   \,+\,P\left(A\cap\bigcup_{n>N}D_n\right).
\end{align*}
If $\omega$ belongs to the last event, no hit has occurred by time $N$, so
$P_N(A)(\omega)\le\delta$ and
$|P_N(A)(\omega)-1|\ge1-\delta$.  Lemma~\ref{lem:au-tail}, used with the smaller
strict threshold $(1-\delta)/2$, implies that the last term tends to zero.  Hence:
\[
 P(A\cap D)=\sum_{n=1}^{\infty}P(A\cap D_n).
\]
Similarly, if
$\omega\in A^c\cap\bigcup_{n>N}D_n$, then $P_m(A)(\omega)>\delta$ for some
$m>N$, so
$\sup_{m>N}|P_m(A)(\omega)-\ind_A(\omega)|>\delta$.  Another application of
Lemma~\ref{lem:au-tail} gives:
\[
 P(A^c\cap D)=\sum_{n=1}^{\infty}P(A^c\cap D_n).
\]
Adding the two identities establishes \eqref{eq:cht}, as desired.
\end{proof}

We close the section by recording three facts about countably additive Borel
probabilities on $\Omega$ that later sections use.  The first two are standard
and we quote them; the third we prove, so that the paper depends on no measure
theory beyond a first course.

\begin{lemma}[L\'evy and Egorov]\label{lem:standard-facts}
Let $\mu$ be a countably additive Borel probability on $\Omega$ that assigns
positive probability to every nonempty cylinder.
\begin{enumerate}[(a),itemsep=.8em,leftmargin=3em,labelsep=1em,topsep=0.6em]
\item For every $B\in\Borel$, the functions
\[
  \omega\longmapsto \mu(B\given[\omega\restrict n])
\]
converge to $\ind_B$ $\mu$-almost surely.
\item If Borel functions $f_n$ converge to $f$ $\mu$-almost surely, then for every
$\rho>0$ there is a Borel event $E$ with $\mu(E)<\rho$ such that $f_n\to f$
uniformly on $E^c$.\qedthm
\end{enumerate}
\end{lemma}

\begin{proof}
For part~(a), fix $B\in\Borel$.  The $\sigma$-algebra $\mathscr F_n$ is generated
by the finite partition of $\Omega$ into length-$n$ cylinders
(\cref{lem:finite-info}), and $\mu$ gives each cell of that partition positive
probability, so the conditional expectation of $\ind_B$ given $\mathscr F_n$ is
the elementary ratio:
\[
  \mathbb{E}_{\mu}\bigl[\ind_B\bigm|\mathscr F_n\bigr](\omega)
  \quad=\quad
  \frac{\mu\bigl(B\cap[\omega\restrict n]\bigr)}
       {\mu\bigl([\omega\restrict n]\bigr)}
  \quad=\quad
  \mu\bigl(B\given[\omega\restrict n]\bigr).
\]
The $\sigma$-algebras $\mathscr F_n$ increase with union generating $\Borel$, so
L\'evy's upward theorem \citep[Theorem~35.6, p.~470]{billingsley1995prob} yields
\[
  \mathbb{E}_{\mu}\bigl[\ind_B\bigm|\mathscr F_n\bigr]
  \quad\longrightarrow\quad
  \mathbb{E}_{\mu}\bigl[\ind_B\bigm|\Borel\bigr]
  \quad=\quad\ind_B
\]
$\mu$-almost surely.  Part~(b) is Egorov's theorem
\citep[Problem~13.9]{billingsley1995prob}.
\end{proof}

\begin{lemma}[Regularity on Cantor space]\label{lem:regularity}
Let $\mu$ be a countably additive Borel probability on $\Omega$.  For every
$A\in\Borel$ and every $\rho>0$, there are a closed set $F$ and an open set $G$
such that $F\subseteq A\subseteq G$ and $\mu(G\setminus F)<\rho$.
\qedthm \end{lemma}

\begin{proof}
Let $\mathscr K$ denote the family of Borel sets with the stated property for
every $\rho>0$.  We show that $\mathscr K$ is a $\sigma$-algebra containing
every closed set; since the closed sets generate $\Borel$, this gives
$\mathscr K=\Borel$.

\emph{Closed sets.}  Let $F_0$ be closed.  If $F_0=\Omega$, then $F=G=\Omega$
serves, so assume $F_0^c\ne\varnothing$.  Its complement is open and therefore a
union of countably many cylinders; let $C_1\subseteq C_2\subseteq\cdots$ be the
partial unions of an enumeration of those cylinders, so that each $C_m$ is clopen
and $\bigcup_mC_m=F_0^c$.  Then each $G_m=C_m^c$ is clopen,
$F_0\subseteq G_m$, and $\bigcap_mG_m=F_0$.  Continuity from above gives
$\mu(G_m\setminus F_0)\to0$, so $F=F_0$ and $G=G_m$ with $m$ large enough serve.

\emph{Complements.}  If $F\subseteq A\subseteq G$ with $\mu(G\setminus F)<\rho$,
then $G^c\subseteq A^c\subseteq F^c$, where $G^c$ is closed and $F^c$ is open,
and $F^c\setminus G^c=G\setminus F$.

\emph{Countable unions.}  Let $A=\bigcup_{j\ge1}A_j$ with every
$A_j\in\mathscr K$, and let $\rho>0$.  Choose closed $F_j$ and open $G_j$ with
$F_j\subseteq A_j\subseteq G_j$ and $\mu(G_j\setminus F_j)<\rho2^{-j-1}$.  The
set $G=\bigcup_jG_j$ is open and contains $A$.  The sets
$H_J=\bigcup_{j\le J}F_j$ are closed and increase to $H=\bigcup_jF_j$, so
continuity from below supplies $J$ with $\mu(H\setminus H_J)<\rho/2$; put
$F=H_J$, a closed subset of $A$.  If $\omega\in G\setminus F$, then
$\omega\in G_j$ for some $j$, and either $\omega\notin F_j$ or
$\omega\in H\setminus H_J$.  Hence:
\[
  G\setminus F
  \quad\subseteq\quad
  \bigcup_{j\ge1}\bigl(G_j\setminus F_j\bigr)\;\cup\;\bigl(H\setminus H_J\bigr),
\]
and therefore $\mu(G\setminus F)<\rho/2+\rho/2=\rho$.
\end{proof}

\begin{lemma}[Clopen approximation]\label{lem:clopen-regular}
Let $\mu$ be a countably additive Borel probability on $\Omega$.  For every
$A\in\Borel$ and every $\rho>0$, there is a clopen $C\in\Clop$ such that
$\mu(A\symdiff C)<\rho$.
\qedthm \end{lemma}

\begin{proof}
By \cref{lem:regularity}, choose a closed $F$ and an open $G$ with
$F\subseteq A\subseteq G$ and $\mu(G\setminus F)<\rho$.  Every point of $F$ lies
in a cylinder contained in $G$, because the cylinders form a neighborhood basis.
The set $F$ is a closed subset of the compact space $\Omega$ and hence compact,
so finitely many of those cylinders cover $F$; their union $C$ is clopen and
satisfies $F\subseteq C\subseteq G$.  Therefore
$A\symdiff C\subseteq G\setminus F$, and so $\mu(A\symdiff C)<\rho$.
\end{proof}

\section{Quantifier Order in the Published Sufficiency Proof}\label{sec:gap}

The published proof of the sufficiency direction of Theorem~2
\citep[pp.~406--407]{nielsen2021convergence} runs as follows.  Let $P$ have the
approximation property and the CHT property, and let $A\in\Borel$ together with
$\eps,\delta\in(0,1)$ be given.  Approximation supplies a clopen event $A_1$,
settled by the first $n_1$ observations, with:
\[
  P(A\symdiff A_1)\;<\;\frac{\eps\delta}{2}.
\]
Write $S_1=A\symdiff A_1$ and $D=\bigcup_{n=1}^{\infty}D_n(S_1,\delta)$.  Each
$D_n(S_1,\delta)$ is a union of length-$n$ cylinders (\cref{lem:finite-info}),
and on each of them the time-$n$ posterior of $S_1$ exceeds $\delta$; summing
over those cylinders gives:
\[
  P\bigl(S_1\cap D_n(S_1,\delta)\bigr)\;\ge\;\delta\,P\bigl(D_n(S_1,\delta)\bigr).
\]
Summing over $n$ and applying CHT to $S_1$ at level $\delta$ yields:
\[
  \frac{\eps\delta}{2}
  \;>\;P(S_1)
  \;\ge\;P(S_1\cap D)
  \;\ge\;\sum_{n=1}^{\infty}P\bigl(S_1\cap D_n(S_1,\delta)\bigr)
  \;\ge\;\delta\,P(D),
\]
so that $P(D)<\eps/2$.  The exceptional event is then $E=S_1\cup D$, which
satisfies $P(E)<\eps$; and for $n\ge n_1$ and $\omega\notin E$ one obtains:
\begin{equation}\label{eq:published-output}
  \bigl|P_n(A)(\omega)-\ind_A(\omega)\bigr|\;\le\;\delta.
\end{equation}

Every step of this argument is correct.  What it delivers, however, is:
\begin{equation}\label{eq:weak-quantifiers}
\begin{split}
  &\forall\eps>0\ \ \forall\delta>0\ \ \exists E\ \ \exists N
   \quad\text{such that }P(E)<\eps\\[2pt]
  &\qquad\text{and \eqref{eq:published-output} holds for all }
   n\ge N\text{ and all }\omega\notin E,
\end{split}
\end{equation}
whereas almost-uniform convergence, by \cref{def:au-general}, requires:
\begin{equation}\label{eq:au-quantifiers}
\begin{split}
  &\forall\eps>0\ \ \exists E\ \ \forall\delta>0\ \ \exists N
   \quad\text{such that }P(E)<\eps\\[2pt]
  &\qquad\text{and \eqref{eq:published-output} holds for all }
   n\ge N\text{ and all }\omega\notin E.
\end{split}
\end{equation}
The event $E=S_1\cup D$ produced by the argument depends on $\delta$ twice over.
The approximation tolerance $\eps\delta/2$ that produces $S_1$ shrinks with
$\delta$, and $D$ is assembled from the hitting events of $S_1$ at level
$\delta$.  Nothing in the argument produces one event that serves every
$\delta$.

Under countable additivity the two prefixes are interchangeable.  Given a
sequence $\delta_m\downarrow0$, apply \eqref{eq:weak-quantifiers} with
probability budget $\eps2^{-m}$ at level $\delta_m$ to obtain events $E_m$, and
put $E=\bigcup_mE_m$.  Countable subadditivity gives $P(E)<\eps$, and $E$
witnesses \eqref{eq:au-quantifiers}, since a bound at level $\delta_m$ is a
bound at every $\delta\ge\delta_m$.  Under finite additivity the union step is
unavailable: monotonicity bounds $P(\bigcup_mE_m)$ from below by each $P(E_m)$
and from above by nothing smaller than $1$.

Whether \eqref{eq:weak-quantifiers} nevertheless implies
\eqref{eq:au-quantifiers} for probability functions with the two criteria is
exactly the content of the sufficiency direction of Theorem~2.
\cref{sec:construction,sec:main} answer the question in the negative, by
constructing a probability function for which the gap between the two prefixes
is realized.

\section{Construction of the Probability Function}\label{sec:construction}

This section builds the probability function that refutes the sufficiency
direction of \eqref{eq:characterization}.

\subsection{What the Construction Must Do}\label{sec:requirements}

The target is a probability function $P$ on $\Borel$ that satisfies
\eqref{eq:posi}, has the approximation property, has the CHT property, and whose
posteriors converge to the truth almost surely but not almost uniformly.  The
last of these requirements pulls against the two before it, and every feature of
the construction is an answer to that tension.

The probability function can be written down before any of its ingredients
are explained:
\begin{align*}
  P\quad&=\quad\alpha\,u\;+\;(1-\alpha)\,R_0,\\[4pt]
  R_0\quad&=\quad\sum_{k=1}^{\infty}t_k
    \Bigl[\,(1-\eta_k)\,\delta_{d_k}\;+\;\eta_k\,\nu_k\,\Bigr].
\end{align*}
The second line is countably additive; everything that fails countable
additivity is in $u$.  The ingredients are built in the order in which they are
needed.  The \textit{distinguished paths} $d_1,d_2,\ldots$, their limit path
$z$, the clopen \textit{shells} $S_1,S_2,\ldots$ that hold them one apiece, and
the diffuse shell measures $\nu_k$ come from \cref{sec:geometry}; the free
ultrafilter and the \textit{diffuse fractions} $\eta_k\in(0,1)$ chosen against
it come from \cref{sec:ultrafilter}; the shell weights $t_k>0$ with
$\sum_kt_k=1$, the point masses $\delta_{d_k}$, and the sum $R_0$ come from
\cref{sec:ca-component}; the two-valued probability $u$, which reads an event's
truth value along $d_1,d_2,\ldots$, comes from \cref{sec:u-component}; and the
mixing weight $\alpha\in(0,1)$ is fixed in \cref{sec:P-companion}, along with
the countably additive measure that clopen events cannot tell from $P$.
\cref{sec:companion-posteriors} turns that measure into a formula for the
posteriors of $P$, and every argument in \cref{sec:main} runs through it.

Two demands bear on the numbers $\eta_k$, and the counterexample lives between
them.  CHT needs, at each fixed threshold, that the indices $k$ whose $\eta_k$
falls below that threshold form a large set.  Almost-uniform convergence needs
one large set of indices along which $\eta_k$ tends to zero.  The construction
must meet the first demand and defeat the second.  No ordinary notion of
largeness will do that, for on the cofinite sets the two demands coincide
(\cref{rem:frechet}).  Nor will an arbitrary free ultrafilter, for passing from
largeness at each threshold to a single large set carrying convergence is a
property some free ultrafilters have and others lack (\cref{rem:ppoint}), and no
one sequence $(\eta_k)$ defeats them all (\cref{rem:frechet}).  The ultrafilter
and the sequence have therefore to be chosen together.

The construction has two stages: a combinatorial device, and a probability
function built on it. Mixing a two-valued probability read along a convergent sequence with a
countably additive measure is the standard way to obtain a merely finitely
additive probability that clopen events cannot detect, and
\citet[proof of Theorem~5, p.~411]{nielsen2021convergence} mixes such a
probability with a single fair-coin measure to separate the approximation
property from almost-sure convergence.  Ours places an atom at each
distinguished path together with a diffuse remainder of weight $\eta_k$ in the
surrounding shell, and the weights $\eta_k$ are chosen against an ultrafilter
that separates threshold-by-threshold smallness from convergence.  Those two
changes convert that separation into a separation of approximation and CHT from
almost-uniform convergence; \cref{rem:nielsen-thm5} sets the two side by
side.

\subsection{Paths, Shells, and Shell Measures}\label{sec:geometry}

The geometry has to supply the paths at which convergence will fail.
Almost-uniform convergence, in the form \eqref{eq:au-quantifiers}, requires one
exceptional event of small probability to serve every accuracy at once.  To
defeat it we need a fixed Borel event together with a supply of paths at which
the posterior errors of that event refuse to become uniformly small, rich enough
that no event of small probability exhausts them.  Those paths need company: a
nonzero posterior error at a path requires the cylinder of its first $k$
observations to carry mass elsewhere than at the path itself, since were the
path the only point of that cylinder charged, every event would already be
settled there and the error would be zero.  So each of the paths is given a
clopen neighborhood of its own, and the remaining mass of that neighborhood is
spread diffusely.

Here is that geometry.  The point
\[
  z\;=\;\bigl(\,0,0,0,\ldots\;\bigr)
\]
is the all-zero path $0^\infty$; it will be the \term{limit path} of the
construction.  For each $k\ge1$, define:
\begin{equation}\label{eq:shellsandpaths}
  S_k\;=\;\bigl[\,0^{k-1}1\,\bigr]
  \qquad \mbox{and} \qquad
  d_k\;=\;0^{k-1}10^\infty.
\end{equation}
Thus $S_k$ is the clopen \term{shell} consisting of paths whose first $1$ occurs
at time $k$, while $d_k$ is the \term{distinguished path} in that shell that is
zero thereafter.
The shells are pairwise disjoint and, together with $\{z\}$, partition $\Omega$:
\begin{equation}\label{eq:shell-partition}
  \Omega\quad=\quad\{z\}\;\mathbin{\dot\cup}\;\bigcup_{k=1}^{\infty}S_k.
\end{equation}
Moreover, $d_k\to z$.  For $n\ge0$, let $Z_{n}$ denote the clopen cylinder of paths that are zero through time $n$:
\begin{equation}\label{eq:clopenzeropaths}
Z_n\quad=\quad\bigl[\,0^n\,\bigr].
\end{equation}
Thus $z\in Z_n$, and $d_k\in Z_n$ if and only if $k>n$.  We write:
\begin{equation}\label{eq:Dset}
  D\quad=\quad\bigl\{\,d_k:k\in\N\,\bigr\}
\end{equation}
for the set of all distinguished paths; it is countable, and each of its points
is isolated in $D\cup\{z\}$, whose only limit point is $z$.

Let $Q$ be the fair-coin measure \eqref{eq:faircoin}.  For each $k$, define
$T_k:\Omega\to S_k$ by concatenation:
\begin{equation}\label{eq:faircoinshift}
  T_k(x)\quad=\quad0^{k-1}1x,
\end{equation}
and let $\nu_k=Q\circ T_k^{-1}$.  Thus $\nu_k$ is fair-coin probability on the
tail coordinates after the first $1$.

\begin{lemma}[The shell measures]\label{lem:nu}
For every $k\in\N$ the map $T_k$ is a continuous injection of $\Omega$ onto
$S_k$, and $\nu_k=Q\circ T_k^{-1}$ is a countably additive Borel probability
such that:
\begin{enumerate}[(a),itemsep=.5em,leftmargin=3em,labelsep=1em,topsep=.6em]
\item $\nu_k(S_k)=1$;
\item $\nu_k$ is diffuse; in particular $\nu_k(D\cup\{z\})=0$, where $D$ is the
      countable set \eqref{eq:Dset};
\item $\nu_k(B)>0$ for every nonempty cylinder $B\subseteq S_k$.\qedthm
\end{enumerate}
\end{lemma}

\begin{proof}
If $T_k(x)=T_k(y)$ then $0^{k-1}1x=0^{k-1}1y$, so $x=y$; and the image of $T_k$
is the set of paths beginning $0^{k-1}1$, which is $S_k$.  If $x$ and $y$ agree
in their first $n$ coordinates, then $T_k(x)$ and $T_k(y)$ agree in their first
$k+n$; so $T_k$ is continuous, hence Borel measurable, and $\nu_k$ is a
countably additive Borel probability.

(a) $T_k^{-1}(S_k)=\Omega$, so $\nu_k(S_k)=Q(\Omega)=1$.

(b) For $\omega\in\Omega$ the set $T_k^{-1}(\{\omega\})$ is empty if
$\omega\notin S_k$, and a single point otherwise by injectivity.  Since $Q$
vanishes on singletons and on $\varnothing$, we get $\nu_k(\{\omega\})=0$.  The
set $D\cup\{z\}$ is countable, so $\nu_k(D\cup\{z\})=0$ by countable
additivity.

(c) A nonempty cylinder $B\subseteq S_k$ has the form $B=[\,0^{k-1}1s\,]$ for a
possibly empty word $s$: its defining word has length at least $k$, since
otherwise $B\not\subseteq S_k$, and it begins $0^{k-1}1$.  Then
$T_k^{-1}(B)=[\,s\,]$ and $\nu_k(B)=Q([\,s\,])=2^{-\abs{s}}>0$.
\end{proof}

\subsection{A Sequence Resisting Diagonalization and Its Ultrafilter}\label{sec:ultrafilter}

The notion of largeness has to separate two things that ordinarily coincide:
smallness at every fixed tolerance on a large set of indices, and smallness at
all tolerances on one common large set.  An ultrafilter supplies the appropriate
notion of ``large.''

\begin{definition}[Ultrafilter]\label{def:ultrafilter}
A \term{filter} on $\N$ is a family $\Ucal\subseteq 2^{\N}$ such that:
\begin{enumerate}[(i),itemsep=0.5em,leftmargin=3em,labelsep=1em,topsep=0.3em]
\item $\varnothing\notin\Ucal$ and $\N\in\Ucal$;
\item if $A,B\in\Ucal$, then $A\cap B\in\Ucal$;
\item if $A\in\Ucal$ and $A\subseteq B\subseteq\N$, then $B\in\Ucal$.
\end{enumerate}
It is an \term{ultrafilter} if, for every $A\subseteq\N$, exactly one of $A$ and
$A^c$ belongs to $\Ucal$.  It is \term{free} if it contains no finite set.
Members of $\Ucal$ will be called \term{$\Ucal$-large}.\qeddef
\end{definition}

A free ultrafilter contains every cofinite subset of $\N$: if $X^c$ is finite,
then $X^c\notin\Ucal$ by freeness, and so $X\in\Ucal$ because an ultrafilter
contains one of the two.  Conversely, an ultrafilter containing every cofinite
set contains no finite set $A$, since it would then contain both $A$ and $A^c$
and hence, by (ii), the empty set.  The existence of a free ultrafilter follows
from the ultrafilter lemma, a weak form of the axiom of choice.  We use no
deeper set theory.

An ultrafilter also supplies a generalized limit for the one kind of sequence
this paper needs.  Let $(a_k)_{k\ge1}$ be a sequence with $a_k\in\{0,1\}$ for
every $k$.  There is exactly one $L\in\{0,1\}$ such that:
\begin{equation}\label{eq:ulimit}
  \bigl\{\,k\in\N:a_k=L\,\bigr\}\in\Ucal;
\end{equation}
we write $L=\lim_{k\to\Ucal}a_k$ and call it the \term{$\Ucal$-limit} of
$(a_k)$.  Existence and uniqueness are immediate: the sets $\{k:a_k=1\}$ and
$\{k:a_k=0\}$ are complementary, so $\Ucal$ contains exactly one of them by
\cref{def:ultrafilter}, and $L$ is the value the contained one records.

The ultrafilter we need is obtained by extending a filter built from a partition
of $\N$ into infinite blocks.  For $i\in\N$ put:
\begin{equation}\label{eq:blocks}
  I_i\quad=\quad\bigl\{\,2^{\,i-1}(2j-1)\;:\;j\in\N\,\bigr\},
\end{equation}
the set of positive integers whose largest power-of-two divisor is $2^{i-1}$.

\begin{lemma}[Blocks]\label{lem:partition}
$\{I_i\}_{i\ge1}$ is a partition of $\N$ into infinitely many infinite sets.
\qedthm \end{lemma}

\begin{proof}
Every $k\in\N$ has a unique representation $k=2^{a}b$ with $a\ge0$ and $b$ odd;
writing $b=2j-1$ places $k$ in $I_{a+1}$ and in no other block, so the $I_i$ are
pairwise disjoint and cover $\N$.  Each $I_i$ is infinite because
$j\mapsto2^{\,i-1}(2j-1)$ is injective.
\end{proof}

\begin{definition}[Block filter]\label{def:blockfilter}
Let:
\begin{equation}\label{eq:blockfilter}
  \Fcal\quad=\quad\Bigl\{\,X\subseteq\N\;:\;
    \bigl\{\,i\in\N:I_i\setminus X\text{ is finite}\,\bigr\}
    \text{ is cofinite}\,\Bigr\}.
\end{equation}
Thus $X\in\Fcal$ when, for all but finitely many blocks, $X$ contains all but
finitely many elements of that block.\qeddef
\end{definition}

The following construction and the argument that any ultrafilter extending
$\Fcal$ has the property we need are due to
\citet[(7.6), p.~76]{jech2003}.

\begin{lemma}[The block filter]\label{lem:blockfilter}
$\Fcal$ is a filter on $\N$ containing every cofinite subset of $\N$, and every
ultrafilter $\Ucal$ with $\Fcal\subseteq\Ucal$ is free.  At least one such
$\Ucal$ exists.
\qedthm \end{lemma}

\begin{proof}
Write $\varphi(X)=\{i:I_i\setminus X\text{ is finite}\}$, so that $X\in\Fcal$
if and only if $\varphi(X)$ is cofinite.

Each $I_i\setminus\varnothing=I_i$ is infinite by \cref{lem:partition}, so
$\varphi(\varnothing)=\varnothing$, which is not cofinite; hence
$\varnothing\notin\Fcal$.  And $\varphi(\N)=\N$, so $\N\in\Fcal$.

Since $I_i\setminus(X\cap Y)=(I_i\setminus X)\cup(I_i\setminus Y)$ and a union
of two finite sets is finite, $\varphi(X)\cap\varphi(Y)\subseteq\varphi(X\cap Y)$;
an intersection of two cofinite sets is cofinite and a superset of a cofinite
set is cofinite, so $X,Y\in\Fcal$ gives $X\cap Y\in\Fcal$.  If
$X\subseteq Y$ then $I_i\setminus Y\subseteq I_i\setminus X$, so
$\varphi(X)\subseteq\varphi(Y)$ and $\Fcal$ is upward closed.

If $\N\setminus X$ is finite then $I_i\setminus X$ is finite for every $i$, so
$\varphi(X)=\N$ and $X\in\Fcal$: every cofinite set belongs to $\Fcal$.
Consequently, if $\Fcal\subseteq\Ucal$ with $\Ucal$ an ultrafilter and
$A\subseteq\N$ is finite, then $\N\setminus A\in\Fcal\subseteq\Ucal$, so
$A\notin\Ucal$; that is, $\Ucal$ is free.  An extension exists by the
ultrafilter lemma.
\end{proof}

For $K\subseteq\N$ and a real sequence $(\eta_k)$, we say that $(\eta_k)$
\term{converges to zero along $K$} in the ordinary subsequence sense, and write
$\eta_k\to0$ along $K$, if for every $r>0$ there is $N\in\N$ such that:
\[
  k\in K,\ k\ge N \quad\Longrightarrow\quad \eta_k<r.
\]
The next lemma separates threshold-by-threshold smallness from ordinary
convergence on one large set.

\begin{lemma}[A sequence resisting diagonalization]\label{lem:diagonal-ultrafilter}
There exist a free ultrafilter $\Ucal$ on $\N$ and a sequence
$(\eta_k)_{k\ge1}\subset(0,1)$ such that:
\begin{enumerate}[(a),itemsep=0.6em,leftmargin=3em,labelsep=1em,topsep=0.8em]
\item for every $r>0$,
\begin{equation}\label{eq:fixed-threshold}
  \{k:\eta_k\le r\}\in\Ucal;
\end{equation}
\item there is no $K\in\Ucal$ along which $\eta_k\to0$.\qedthm
\end{enumerate}
\end{lemma}

\begin{proof}
Let $\Ucal$ be any ultrafilter with $\Fcal\subseteq\Ucal$, which is free by
\cref{lem:blockfilter}, and set:
\begin{equation}\label{eq:etadef}
  \eta_k\quad=\quad\frac{1}{i+1}
  \qquad\text{for the unique }i\text{ with }k\in I_i .
\end{equation}
This is unambiguous by \cref{lem:partition}, and
$\eta_k\in\{\tfrac12,\tfrac13,\tfrac14,\ldots\}\subset(0,1)$.  Thus $\eta$ is
constant on each block, and the block values decrease to zero as the block
index grows.

For \eqref{eq:fixed-threshold}, let $r>0$ and choose $M$ with $1/(M+1)\le r$.
Put $T_M=\bigcup_{i\ge M}I_i$.  For $i\ge M$ we have
$I_i\setminus T_M=\varnothing$, so $\varphi(T_M)\supseteq\{i:i\ge M\}$, a
cofinite set; hence $T_M\in\Fcal\subseteq\Ucal$.  Every $k\in T_M$ lies in some
$I_i$ with $i\ge M$ and so has $\eta_k=1/(i+1)\le1/(M+1)\le r$; that is,
$T_M\subseteq\{k:\eta_k\le r\}$, and upward closure of $\Ucal$ gives
\eqref{eq:fixed-threshold}.

For part~(b), suppose $K\in\Ucal$ and $\eta_k\to0$ along $K$.  Fix $i\in\N$ and
apply the definition of convergence along $K$ with $r=1/(i+1)$: there is $N$
such that every $k\in K$ with $k\ge N$ satisfies $\eta_k<1/(i+1)$.  No element
of $I_i$ satisfies that strict inequality, since $\eta$ equals $1/(i+1)$
throughout $I_i$; hence $K\cap I_i\subseteq\{1,\ldots,N-1\}$ is finite.  As $i$
was arbitrary, $I_i\setminus(\N\setminus K)=I_i\cap K$ is finite for every $i$,
so $\varphi(\N\setminus K)=\N$ and $\N\setminus K\in\Fcal\subseteq\Ucal$.
Together with $K\in\Ucal$ this puts $\varnothing$ in $\Ucal$, which is
impossible.
\end{proof}

A set is $\Ucal$-large as soon as it exhausts all but finitely many blocks up to
finitely many points.  So no $\Ucal$-large set can meet every block
finitely, and every $\Ucal$-large set therefore contains infinitely many points
of some single block, on which $\eta$ is constant and positive.

\begin{remark}[Why an ultrafilter, and why it must be chosen]\label{rem:frechet}
Neither part of \cref{lem:diagonal-ultrafilter} survives a more familiar notion
of largeness.  Suppose ``large'' meant ``cofinite,'' so that $\Ucal$ were
replaced by the Fr\'echet filter.  Then part~(a) would say that
$\{k:\eta_k\le r\}$ is cofinite for every $r>0$, which is exactly $\eta_k\to0$,
and part~(b) would fail because $\N$ is itself cofinite.  Nor can the sets in
part~(a) be closed under countable intersection: since $\eta_k>0$ for every $k$,
the sets $\{k:\eta_k\le1/m\}$ have empty intersection, and no filter contains
$\varnothing$.  What the lemma needs is a notion of largeness closed under finite
but not countable intersection --- the same gap between finite and countable
additivity that the counterexample exploits.

The same computation shows that the ultrafilter cannot be left unspecified.
First, a subset of $\N$ belonging to \emph{every} free ultrafilter is cofinite.
For suppose $Y\subseteq\N$ is not cofinite, so that $\N\setminus Y$ is infinite.
Then $\N\setminus Y$ meets every cofinite set, so the family
\[
  \bigl\{\,Z\subseteq\N:Z\supseteq(\N\setminus Y)\cap C
    \mbox{ for some cofinite }C\,\bigr\}
\]
contains no empty set and is closed under intersection and under passage to
supersets; it is therefore a filter, and it contains every cofinite set as well
as $\N\setminus Y$.  Any ultrafilter extending it is free and omits $Y$.  Now
suppose a single sequence $(\eta_k)\subset(0,1)$ satisfied part~(a) of
\cref{lem:diagonal-ultrafilter} against every free ultrafilter.  Then
$\{k:\eta_k\le r\}$ would belong to every free ultrafilter, hence be cofinite,
for every $r>0$; that is $\eta_k\to0$, and part~(b) would then fail against
every free ultrafilter, taking $K=\N$.  So no sequence serves them all, and the
pair $\bigl(\Ucal,(\eta_k)\bigr)$ must be selected rather than assumed --- which
is what \cref{lem:blockfilter} and \cref{lem:diagonal-ultrafilter} do.
\qedthm\end{remark}

\begin{remark}[No P-point will serve]\label{rem:ppoint}
The choice is constrained further, though nothing below depends on it.  Let
$\set{B_m}_{m\ge1}$ be a countable family of subsets of $\N$.  A set
$K\subseteq\N$ is a \term{pseudo-intersection} of the family if
$K\setminus B_m$ is finite for every $m\ge1$; only \emph{almost} inclusion is
asked, and the genuine intersection $\bigcap_{m\ge1}B_m$ is typically empty
even when a pseudo-intersection is infinite.  A free ultrafilter $\Ucal$ on $\N$
is a \term{P-point} if every countable family of $\Ucal$-large sets has a
pseudo-intersection that is itself $\Ucal$-large; this is one of several
equivalent formulations of the notion
\citep[Exercise~7.7, p.~86]{jech2003}.

No P-point admits a sequence of the kind \cref{lem:diagonal-ultrafilter}
produces.  Suppose $\Ucal$ were a P-point and $(\eta_k)\subset(0,1)$ satisfied
part~(a) of that lemma.  The sets $B_m=\set{k:\eta_k\le1/m}$ are then
$\Ucal$-large for every $m\ge1$, so there is a $\Ucal$-large $K$ with
$K\setminus B_m$ finite for every $m$.  Let $r>0$, choose $m$ with $1/m<r$, and
choose $N$ with $K\setminus B_m\subseteq\{1,\ldots,N-1\}$.  Every $k\in K$ with
$k\ge N$ then lies in $B_m$, so $\eta_k\le1/m<r$.  Hence $\eta_k\to0$ along
$K\in\Ucal$, contradicting part~(b).

Nor is the property available for free.  Whether P-points exist is independent
of the usual axioms of set theory.  They exist under the continuum hypothesis,
since every Ramsey ultrafilter is a P-point and a Ramsey ultrafilter exists when
$2^{\aleph_0}=\aleph_1$ \citep[Theorem~7.8, pp.~76--77]{jech2003}; and there is
a model of $\mathsf{ZFC}$ containing no P-point at all, a result of Shelah
recorded by \citet[p.~89]{jech2003} and published in detail by
\citet{wimmers1982}.  \cref{lem:diagonal-ultrafilter} sidesteps the question,
arranging what is needed outright from \cref{lem:blockfilter} and the
ultrafilter lemma.
\qedthm\end{remark}

\begin{remark}[On the word ``diagonal'']\label{rem:diagonal-terminology}
Diagonalization is the standard passage from a family of conditions, each
satisfiable on its own, to a single object satisfying all of them at once: the
diagonal subsequence of analysis, and the countable union of \cref{sec:gap}.
\cref{lem:diagonal-ultrafilter} is the failure of that passage: the bound
$\eta_k\le r$ holds for each $r$ separately, and along no one $\Ucal$-large set
for all $r$ at once.  We therefore speak throughout of resisting
diagonalization, and never of a diagonal sequence.  In analysis a diagonal
sequence is one that the passage succeeds in extracting, which is the opposite
of what \cref{lem:diagonal-ultrafilter} supplies.  The adjective is retained
for two notions only, both of \cref{sec:diagonal}: a \textit{diagonal
condition} is one imposed simultaneously at all thresholds, and so a condition
demanding the very passage that fails here, and a \textit{diagonal
strengthening} of CHT is a strengthening of CHT by such a condition.  Two
clarifications follow: the first concerns a notion that must be kept distinct
from this one, the second a
notion that is exactly
this one and that our $\Ucal$ must fail.

First, a filter on a cardinal $\kappa$ is \textit{normal} when it is closed under
diagonal intersections $\bigtriangleup_{\alpha<\kappa}X_\alpha
=\{\xi<\kappa:\xi\in\bigcap_{\alpha<\xi}X_\alpha\}$
\citep[pp.~91, 95]{jech2003}, a notion developed for regular uncountable
$\kappa$ and, in the measure-theoretic form, for measurable cardinals
\citep[p.~131]{jech2003}.  Our $\Ucal$ lives on $\N$ and is not even closed
under countable intersections, so it is not normal in that sense and nothing
here bears on normality.

Second, and more to the point, the P-point property of \cref{rem:ppoint} --- a
pseudo-intersection in $\Ucal$ for every countable family of $\Ucal$-large sets
--- is diagonalization of that family in the set-theoretic sense.  The
ultrafilter we need must \emph{fail} to have it: part~(b) of
\cref{lem:diagonal-ultrafilter} is precisely a failure of pseudo-intersection.
So ``diagonal'' attributes no diagonalization property to $\Ucal$, and by
\cref{rem:ppoint} no P-point would serve in its place.
\qedthm\end{remark}

\subsection{The Countably Additive Component}\label{sec:ca-component}

The countably additive component has to keep the posteriors at the distinguished
paths under control at every later time.  With the geometry fixed and the
numbers $\eta_k$ in hand, we assemble it: an atom at each distinguished path,
together with a diffuse remainder of relative weight $\eta_k$ in the
surrounding shell.

Let $(\Ucal,(\eta_k))$ be supplied by \cref{lem:diagonal-ultrafilter}, put
$t_k=2^{-k}$, and define:
\begin{equation}\label{eq:R0}
  R_0
  \quad=\quad\sum_{k=1}^{\infty}t_k
    \biggl((1-\eta_k)\delta_{d_k}+\eta_k\nu_k\biggr).
\end{equation}
Because $\sum_kt_k=1$, this is a countably additive Borel probability.  Its total
mass on $S_k$ is $t_k$, of which $t_k(1-\eta_k)$ is concentrated at $d_k$ and
$t_k\eta_k$ is diffuse.  We call $\eta_k$ the \term{diffuse fraction} of the
$k$th shell; it is exactly the posterior error at $d_k$ at time $k$, computed
in \eqref{eq:error-eta}.

The atom at $d_k$ is what keeps those posteriors under control.
Making the time-$k$ errors large is easy; keeping CHT while doing so is not.
CHT requires the infinite hitting union $O(A,\delta)$ to receive the limit of
the probabilities of its finite pieces, for every event $A$ and every threshold
$\delta$.  The finite pieces are clopen, and \cref{sec:P-companion} will hand
them to a countably additive companion measure; what remains is to control which
distinguished paths ever hit level $\delta$, and for that a bound on
$P\bigl(A\given[d_k\restrict n]\bigr)$ at the single time $n=k$ is of no use.
The bound must hold at every $n\ge k$ at once.  An atom at $d_k$ delivers one:
it places a floor under the denominator that later conditioning cannot erode,
while an event omitting $d_k$ can draw on nothing but the diffuse remainder in
the numerator.  Nor is the atom a convenience that another design might avoid:
by \cref{prop:noatom}, no countably additive component leaving the distinguished
paths unatomized has the CHT property at all.  We record the bound in the
generality in which it is used.

\begin{lemma}[Atoms bound later posteriors]\label{lem:atom}
Let $\mu$ be a countably additive Borel probability on $\Omega$, let $v$ be a
probability function on $\Borel$, let $\beta\in(0,1)$, and put
$P_v=\beta v+(1-\beta)\mu$.  Let $d\in\Omega$ and $n_0\in\N$, and suppose
\begin{enumerate}[(i),itemsep=.35em,leftmargin=2.8em,labelsep=.8em,topsep=.5em]
\item $\mu(\{d\})>0$, and
\item $v\bigl([d\restrict n]\bigr)=0$ for every $n\ge n_0$.
\end{enumerate}
Then for every $A\in\Borel$ with $d\notin A$ and every $n\ge n_0$:
\[
  P_v\bigl(A\given[d\restrict n]\bigr)
  \;\le\;\frac{\gamma_n}{\mu(\{d\})+\gamma_n},
  \qquad\text{where }
  \gamma_n=\mu\bigl([d\restrict n]\setminus\{d\}\bigr).
\]
\qedthm \end{lemma}

\begin{proof}
Let $n\ge n_0$ and let $A\in\Borel$ with $d\notin A$.  Then
$A\cap[d\restrict n]\subseteq[d\restrict n]\setminus\{d\}$, so
$\mu(A\cap[d\restrict n])\le\gamma_n$ by monotonicity; and
$v(A\cap[d\restrict n])\le v([d\restrict n])=0$ by (ii) and monotonicity.  Hence
the numerator of $P_v(A\given[d\restrict n])$ is at most $(1-\beta)\gamma_n$.
Since $\{d\}$ and $[d\restrict n]\setminus\{d\}$ partition $[d\restrict n]$,
the denominator satisfies
$P_v([d\restrict n])\ge(1-\beta)\mu([d\restrict n])
=(1-\beta)(\mu(\{d\})+\gamma_n)$.  Dividing gives the bound.
\end{proof}

\subsection{The Ultrafilter Component}\label{sec:u-component}

The finitely additive component has to be invisible to clopen events.  Both of
the criteria at issue are conditions about clopen and open events: approximation
asks that every Borel event lie within $\eps$, in $P$-symmetric difference, of a
clopen event, and CHT asks a continuity property of the open hitting unions.  A
countably additive probability with both properties converges almost uniformly,
because the argument of \cref{sec:gap} is then valid; so $P$ must be merely
finitely additive, and its finitely additive part must be one that no clopen
event detects.  A two-valued probability read along the distinguished paths is
such a part.  Since $d_k\to z$, a clopen event has the same truth value at $d_k$
for all large $k$ as it has at $z$, so on clopen events such a probability
agrees with the point mass at $z$; on the countable set $D$ of \eqref{eq:Dset},
by contrast, it takes the value $1$ while giving each singleton $\{d_k\}$ the
value $0$, which is precisely the failure of countable additivity.

Formally, the component regards a Borel event as true when the indices of the
distinguished paths it contains form a $\Ucal$-large set.  Define:
\begin{equation}\label{eq:u}
  u(A)\quad=\quad
  \begin{cases}
    1 &\mbox{if }\;\{k:d_k\in A\}\in\Ucal;\\
    0&\mbox{if }\;\{k:d_k\in A\}\notin\Ucal,
  \end{cases}
  \qquad A\in\Borel.
\end{equation}

\begin{lemma}[Elementary properties of the ultrafilter probability]
\label{lem:u-properties}
The set function $u$ is a two-valued finitely additive probability.  Moreover,
for all Borel $A,B$:
\begin{enumerate}[(a),itemsep=0.6em,leftmargin=3em,labelsep=1em,topsep=0.8em]
\item $u(A\cap B)=u(A)u(B)$;
\item if $u(A)=u(B)$, then $u(A\symdiff B)=0$.\qedthm
\end{enumerate}
\end{lemma}

\begin{proof}
Normalization follows from $\N\in\Ucal$.  If $A$ and $B$ are disjoint, the index
sets $I_A=\{k:d_k\in A\}$ and $I_B=\{k:d_k\in B\}$ are disjoint.  An ultrafilter
can contain at most one of them.  If it contains neither, then it contains both
complements and hence the complement of $I_A\cup I_B$; thus it does not contain
the union.  Therefore $u(A\cup B)=u(A)+u(B)$.

For part~(a), if $u(A)=u(B)=1$, then $I_A\cap I_B\in\Ucal$; if either value is
zero, then $I_A\cap I_B$ is contained in a set outside $\Ucal$ and hence is
outside $\Ucal$.  For part~(b), if both values are one then
$I_A\cap I_B\in\Ucal$, while if both are zero then
$I_A^c\cap I_B^c\in\Ucal$.  In either case the complement of the index set for
$A\symdiff B$ belongs to $\Ucal$.
\end{proof}

\subsection{The Probability Function and Its Companion}\label{sec:P-companion}

With both components in hand we can assemble $P$, and with it the countably
additive measure that clopen events cannot distinguish from it.  Fix
$\alpha\in(0,1)$ and set:
\begin{equation}\label{eq:P}
  P\;=\;\alpha u\,+\,(1-\alpha)R_0.
\end{equation}
The countably additive probability
\begin{equation}\label{eq:R-companion}
  R\;=\;\alpha\delta_z+(1-\alpha)R_0
\end{equation}
will serve as the \term{companion measure} of $P$.  The reason for introducing
$R$ is that clopen
events cannot distinguish the generalized limit $u$ from ordinary point mass at
$z$: since $d_k\to z$, membership of $d_k$ in a clopen set is eventually the same
as membership of $z$.

\begin{proposition}\label{prop:basic-properties}
The set function $P$ is a merely finitely additive probability.  Moreover:
\begin{equation}\label{eq:clopen-agreement}
  P(C)\;=\; R(C) \qquad \mbox{for all }C\in\Clop,
\end{equation}
and every nonempty cylinder has strictly positive $P$-probability.
\qedthm \end{proposition}

\begin{proof}
\emph{Failure of countable additivity.}
Let $D$ be the set of distinguished paths \eqref{eq:Dset}.  Since
$\{k:d_k\in D\}=\N\in\Ucal$, we have $u(D)=1$; and $u(\{d_k\})=0$ for every $k$,
because a free ultrafilter contains no singleton.  Thus:
\begin{align*}
  P(D)
    \quad&=\quad\alpha\;+\;(1-\alpha)\sum_{k=1}^{\infty}t_k(1-\eta_k),\\
  \sum_{k=1}^{\infty}P(\{d_k\})
    \quad&=\quad(1-\alpha)\sum_{k=1}^{\infty}t_k(1-\eta_k).
\end{align*}
The two quantities differ by $\alpha$, so $P$ is not countably additive.

\emph{Agreement on clopen events.}
If $C$ is clopen, continuity of $\ind_C$ and the convergence $d_k\to z$ imply
that $\ind_C(d_k)$ is eventually equal to $\ind_C(z)$.  Every cofinite set belongs
to $\Ucal$, so:
\[
  u(C)\;=\;\delta_z(C).
\]
Substitution into \eqref{eq:P} and \eqref{eq:R-companion} gives
\eqref{eq:clopen-agreement}.

\emph{Positivity.}
Let $B$ be a nonempty cylinder.  If $z\in B$, then
$P(B)=R(B)\ge\alpha$.  If $z\notin B$, there is a unique $k$ such that
$B\subseteq S_k$.  Since $\nu_k(B)>0$, it follows that:
\[
  P(B)\;=\;R(B)\;\ge\;(1-\alpha)t_k\eta_k\nu_k(B)>0.
\]
\end{proof}

Applied to the construction, \cref{lem:atom} yields the estimate used
throughout \cref{sec:main} and \cref{sec:positive}.  Only two features of the
data enter the proof --- that the diffuse fractions lie in $(0,1)$ and that the
ultrafilter is free --- so we state the estimate for an arbitrary sequence of
fractions and an arbitrary free ultrafilter, and read off the two cases we need.

\begin{corollary}[The shell estimate]\label{cor:shell}
Let $\Vcal$ be a free ultrafilter on $\N$, let $u_{\Vcal}$ be defined from
$\Vcal$ by \eqref{eq:u}, let $(\lambda_k)_{k\ge1}\subset(0,1)$, and put:
\[
  R_0^{\lambda}
  \;=\;\sum_{k=1}^{\infty}t_k
    \bigl((1-\lambda_k)\delta_{d_k}+\lambda_k\nu_k\bigr),
  \qquad
  P^{\lambda}\;=\;\alpha u_{\Vcal}+(1-\alpha)R_0^{\lambda}
  \quad\mbox{with }\alpha\in(0,1).
\]
Then for every $k\in\N$, every $A\in\Borel$ with $d_k\notin A$, and every
$n\ge k$:
\begin{equation}\label{eq:shell-estimate}
  P^{\lambda}\bigl(A\given[d_k\restrict n]\bigr)\;\le\;\lambda_k .
\end{equation}
In particular the bound holds for the probability function $P$ of \eqref{eq:P},
with $\lambda_k=\eta_k$, and for the probability functions of
\cref{sec:positive}, with $\lambda_k=\theta_k$.
\qedthm \end{corollary}

\begin{proof}
Fix $k$ and $n\ge k$.  The paths $d_1,d_2,\ldots$ are distinct, and each $\nu_j$
is diffuse by \cref{lem:nu}(b), so
\[
  R_0^{\lambda}(\{d_k\})
  =\sum_{j\ge1}t_j\bigl[(1-\lambda_j)\delta_{d_j}(\{d_k\})
     +\lambda_j\nu_j(\{d_k\})\bigr]
  =t_k(1-\lambda_k)>0 .
\]
Since $n\ge k$, the cylinder $[d_k\restrict n]=[\,0^{k-1}1\,0^{\,n-k}\,]$ is
contained in $S_k$, and only the $j=k$ term of $R_0^{\lambda}$ charges it:
$\delta_{d_j}$ does so only for $j=k$, because the defining word places the
first $1$ in position $k$; and $\nu_j$ does so only for $j=k$, because
$\nu_j(S_j)=1$ by \cref{lem:nu}(a) while the shells are pairwise disjoint.
Using diffuseness once more,
\[
  \gamma_n=R_0^{\lambda}\bigl([d_k\restrict n]\setminus\{d_k\}\bigr)
  =t_k\lambda_k\,\nu_k\bigl([d_k\restrict n]\bigr)=t_k\lambda_kq,
  \qquad q=\nu_k\bigl([d_k\restrict n]\bigr)\in[0,1].
\]
Also $\{j:d_j\in[d_k\restrict n]\}=\{k\}$ is finite, so
$u_{\Vcal}([d_k\restrict n])=0$ because $\Vcal$ is free.  \cref{lem:atom},
applied with $\mu=R_0^{\lambda}$, $v=u_{\Vcal}$, $\beta=\alpha$, $d=d_k$ and
$n_0=k$, therefore gives
\[
  P^{\lambda}\bigl(A\given[d_k\restrict n]\bigr)
  \;\le\;\frac{t_k\lambda_kq}{t_k(1-\lambda_k)+t_k\lambda_kq}
  \;=\;\frac{\lambda_kq}{1-\lambda_k+\lambda_kq}
  \;\le\;\lambda_k ,
\]
the last inequality because
$\lambda_kq\le\lambda_k(1-\lambda_k+\lambda_kq)$ reduces to
$q(1-\lambda_k)\le1-\lambda_k$, which holds as $q\le1$ and $\lambda_k<1$.
\end{proof}

\subsection{Posteriors Under the Companion Measure}\label{sec:companion-posteriors}

The next lemma is the bridge back to ordinary countably additive probability.
For each Borel event $A$, we alter only whether the point $z$ belongs to $A$, so
that its membership agrees with the ultrafilter value of $A$.  The resulting
event has, under a countably additive companion, exactly the posterior process
that $A$ has under the finitely additive probability.  Only two features of the
countably additive component enter, so we state the lemma for an arbitrary one
giving $z$ probability zero, and for an arbitrary free ultrafilter;
\cref{sec:positive} uses it in that generality.

\begin{lemma}[Countably additive companion]\label{lem:companion}
Let $\Vcal$ be a free ultrafilter on $\N$, let $u_{\Vcal}$ be defined from
$\Vcal$ by \eqref{eq:u}, let $\mu_0$ be a countably additive Borel probability
on $\Omega$ with $\mu_0(\{z\})=0$, let $\alpha\in(0,1)$, and put:
\begin{equation}\label{eq:companion-pair}
  P'\;=\;\alpha u_{\Vcal}+(1-\alpha)\mu_0,
  \qquad
  R'\;=\;\alpha\delta_z+(1-\alpha)\mu_0 .
\end{equation}
For $A\in\Borel$, define:
\begin{equation}\label{eq:Asharp}
 A^{\sharp}\quad=\quad
 \begin{cases}
   A\cup\{z\}&\mbox{if }u_{\Vcal}(A)=1;\\
   A\setminus\{z\}&\mbox{if }u_{\Vcal}(A)=0,
 \end{cases}
\end{equation}
the dependence on $\Vcal$ being suppressed in the notation.  Then, for every
clopen $C$:
\begin{equation}\label{eq:companion-intersection}
  P'(A\cap C)\;=\;R'(A^{\sharp}\cap C);
\end{equation}
in particular $P'(C)=R'(C)$.  If moreover $P'$ satisfies \eqref{eq:posi}, then:
\begin{equation}\label{eq:posterior-companion}
  P'_n(A)(\omega)
  \;=\;R'\bigl(A^{\sharp}\given[\omega\restrict n]\bigr)
\end{equation}
for every $\omega\in\Omega$ and every $n\in\N$.  Taking $\Vcal=\Ucal$ and
$\mu_0=R_0$ gives $P'=P$ and $R'=R$.\qedthm
\end{lemma}

\begin{proof}
Let $C$ be clopen.  Since $\ind_C$ is continuous and $d_k\to z$, we have
$\ind_C(d_k)=\ind_C(z)$ for all large $k$; that index set is cofinite, hence
$\Vcal$-large, so $u_{\Vcal}(C)=\ind_C(z)$.  Multiplicativity of the two-valued
probability $u_{\Vcal}$ (\cref{lem:u-properties}(a), whose proof uses nothing
about the ultrafilter beyond its being one) therefore gives
\[
  u_{\Vcal}(A\cap C)=u_{\Vcal}(A)\,\ind_C(z)
    =\ind_{A^{\sharp}}(z)\,\ind_C(z)
    =\delta_z\bigl(A^{\sharp}\cap C\bigr),
\]
the middle step because $\ind_{A^{\sharp}}(z)=1$ exactly when
$u_{\Vcal}(A)=1$, by \eqref{eq:Asharp}.  Also
$A\symdiff A^{\sharp}\subseteq\{z\}$ and $\mu_0(\{z\})=0$, so
$\mu_0(A\cap C)=\mu_0(A^{\sharp}\cap C)$.  Combining the two with weights
$\alpha$ and $1-\alpha$ gives \eqref{eq:companion-intersection}, and $A=\Omega$
gives $P'(C)=R'(C)$.  Under \eqref{eq:posi} the common denominator
$P'([\omega\restrict n])=R'([\omega\restrict n])$ is positive, and division
yields \eqref{eq:posterior-companion}.
\end{proof}

Every question about the posteriors generated by $P$ is thereby converted into a
question about ordinary conditional probabilities under a countably additive
measure, to which the results of \cref{sec:necessity} apply.

The construction is complete.  For reference in \cref{sec:main}, its
ingredients are: the ultrafilter $\Ucal$ and the sequence $(\eta_k)$ of
\cref{lem:diagonal-ultrafilter}; the limit path $z$, the shells $S_k$, the
distinguished paths $d_k$ and their set $D$, and the zero-cylinders $Z_n$, from
\eqref{eq:shellsandpaths}, \eqref{eq:clopenzeropaths}, and \eqref{eq:Dset}; the
weights $t_k=2^{-k}$ and the diffuse shell measures $\nu_k$ of
\eqref{eq:faircoinshift}; the countably additive probability $R_0$ of
\eqref{eq:R0}; the ultrafilter probability $u$ of \eqref{eq:u}; the probability
function $P$ of \eqref{eq:P}; its countably additive companion $R$ of
\eqref{eq:R-companion}; and the modified event $A^{\sharp}$ of
\eqref{eq:Asharp}.

\section{Approximation and CHT Hold; Almost-Uniform Convergence Fails}\label{sec:main}

We show that $P$ satisfies both of Nielsen's criteria, that its posteriors
nevertheless fail to converge almost uniformly, and that they do converge almost
surely.

\subsection{Approximation}

The first criterion reduces to clopen approximation for the countably additive
component, together with one adjustment near the limit path $z$ that gives the
approximating clopen event the correct ultrafilter value.

\begin{proposition}\label{prop:approximation}
The probability function $P$ has the approximation property.
\qedthm \end{proposition}

\begin{proof}
Fix $A\in\Borel$ and $\eps>0$, and put $b=u(A)\in\{0,1\}$.  Choose:
\[
  0<\rho<\frac{\eps}{2(1-\alpha)}.
\]
By \cref{lem:clopen-regular}, there is a clopen $C_0$ with
$R_0(A\symdiff C_0)<\rho$.  This already controls the countably additive part;
we modify $C_0$ near the limit path $z$ so that it also has the correct
ultrafilter value.

The zero-cylinders $Z_n=[0^n]$ of \eqref{eq:clopenzeropaths} satisfy, by
\eqref{eq:shell-partition} and \eqref{eq:R0}:
\begin{equation}\label{eq:R0Zn}
  R_0(Z_n)=\sum_{k>n}t_k=2^{-n}.
\end{equation}
Choose $n$ so that $R_0(Z_n)<\rho$.  If $b=1$, set $C=C_0\cup Z_n$; if
$b=0$, set $C=C_0\setminus Z_n$.  In the first case $C$ contains all $d_k$ with
$k>n$, and in the second it contains none of them.  Hence $u(C)=b=u(A)$, and
Lemma~\ref{lem:u-properties}(b) gives $u(A\symdiff C)=0$.  Moreover:
\[
  R_0(A\symdiff C)
  \le R_0(A\symdiff C_0)+R_0(Z_n)<2\rho.
\]
Therefore:
\[
  P(A\symdiff C)<2(1-\alpha)\rho<\eps.
\]
\end{proof}

\subsection{Conditional Hitting Times}

The construction was designed so that finite-stage hitting events are easy: they
are clopen by \cref{lem:finite-info}, and $P$ agrees with the countably additive
measure $R$ on every clopen set.  The only issue is the infinite hitting union.  We show that its ultrafilter
value is $1$ exactly when it contains the limit point $z$.

\begin{proposition}\label{prop:cht}
The probability function $P$ has the CHT property.
\qedthm \end{proposition}

\begin{proof}
Fix $A\in\Borel$ and $\delta\in(0,1)$, and abbreviate
$C_n=C_n(A,\delta)$,
$O_N=\bigcup_{n\le N}C_n$, and $O=\bigcup_{n\ge1}C_n$.  Every $O_N$ is
clopen, so Proposition~\ref{prop:basic-properties} and countable additivity of $R$
give:
\begin{equation}\label{eq:finite-hit-limit}
  \lim_{N\to\infty}P(O_N)
   =\lim_{N\to\infty}R(O_N)=R(O).
\end{equation}
Since $P-R=\alpha(u-\delta_z)$, it remains to prove:
\begin{equation}\label{eq:open-agreement}
  u(O)\quad=\quad\ind_O(z).
\end{equation}

\smallskip
\noindent\emph{Case 1: $z\in O$.}
Openness gives some $m$ with $Z_m\subseteq O$, where $Z_m=[0^m]$ is as in
\eqref{eq:clopenzeropaths}.  Since $d_k\in Z_m$ for every $k>m$, the index set
$\{k:d_k\in O\}$ is cofinite, so $u(O)=1=\ind_O(z)$.

\smallskip
\noindent\emph{Case 2: $z\notin O$.}
Then
\begin{equation}\label{eq:z-no-hit}
  P(A\given Z_n)\le\delta \qquad(n\in\N).
\end{equation}
Because $\{k:d_k\in Z_n\}=\{k:k>n\}$ is cofinite, $u(Z_n)=1$, and Lemma~\ref{lem:u-properties}(a) gives
$u(A\cap Z_n)=u(A)$.  Also
$R_0(A\cap Z_n)\le R_0(Z_n)\to0$.  Hence
\[
 P(A\given Z_n)
 =\frac{\alpha u(A)+(1-\alpha)R_0(A\cap Z_n)}
        {\alpha+(1-\alpha)R_0(Z_n)}
 \longrightarrow u(A).
\]
Since $u(A)$ is either $0$ or $1$ and $\delta<1$, \eqref{eq:z-no-hit} implies
$u(A)=0$.  Hence
\begin{equation}\label{eq:IA-small}
  I_A:=\set{k:d_k\in A}\notin\Ucal.
\end{equation}

Fix $k\notin I_A$.  If $n<k$, then
$[d_k\restrict n]=Z_n$, so \eqref{eq:z-no-hit} shows that $d_k\notin C_n$.
If $n\ge k$, then $d_k\notin A$ and \eqref{eq:shell-estimate} give
$P(A\given[d_k\restrict n])\le\eta_k$, so $d_k\notin C_n$ whenever
$\eta_k\le\delta$.  It follows that
\begin{equation}\label{eq:hit-index-inclusion}
  \set{k:d_k\in O}
  \subseteq I_A\cup\set{k:\eta_k>\delta}.
\end{equation}
The first set on the right is not in $\Ucal$ by \eqref{eq:IA-small}; the second is
not in $\Ucal$ by Lemma~\ref{lem:diagonal-ultrafilter}(a).  A finite union of sets
outside an ultrafilter is again outside it.  Thus $u(O)=0$, proving
\eqref{eq:open-agreement}.

Combining \eqref{eq:finite-hit-limit} and \eqref{eq:open-agreement} gives
$P(O)=\lim_NP(O_N)$, which is equivalent to CHT by
\eqref{eq:cht-continuity}.
\end{proof}

The atom at each distinguished path was what made \eqref{eq:shell-estimate}
available, and with it the whole of Case~2.  It cannot be dispensed with.  The
next proposition holds the geometry of \cref{sec:geometry} and the ultrafilter
component fixed and lets the countably additive component vary: if that
component gives no mass to the distinguished paths, on a large set of indices,
then CHT fails outright.  The event that witnesses the failure is the one that
will witness the failure of almost-uniform convergence as well:
\begin{equation}\label{eq:H}
  H\quad=\quad\Omega\setminus\bigl(D\cup\{z\}\bigr),
\end{equation}
where $D$ is the set of distinguished paths \eqref{eq:Dset}.  That set is
countable and every singleton is closed, so $H$ is Borel.

\begin{proposition}[No atoms, no CHT]\label{prop:noatom}
Let $\Vcal$ be a free ultrafilter on $\N$, let $u_{\Vcal}$ be defined from
$\Vcal$ by \eqref{eq:u}, let $\mu_0$ be a countably additive Borel probability
on $\Omega$ with $\mu_0(\{z\})=0$, let $\alpha\in(0,1)$, and suppose that
$P'=\alpha u_{\Vcal}+(1-\alpha)\mu_0$ satisfies \eqref{eq:posi}.  If:
\begin{equation}\label{eq:noatoms}
  J_0\;=\;\bigl\{\,k\in\N:\mu_0(\{d_k\})=0\,\bigr\}\;\in\;\Vcal,
\end{equation}
then $P'$ does not have the CHT property.  Indeed, for $H$ as in
\eqref{eq:H} and every $\delta\in(1-\alpha,1)$,
\[
  P'\bigl(O(H,\delta)\bigr)-\lim_{N\to\infty}
  P'\bigl(O_N(H,\delta)\bigr)\;=\;\alpha .
\]
\qedthm \end{proposition}

\begin{proof}
Fix $\delta\in(1-\alpha,1)$ and abbreviate $O=O(H,\delta)$,
$O_N=O_N(H,\delta)$.

\smallskip
\noindent\emph{Step 1: the limit path never crosses.}
No distinguished path lies in $H$, so $u_{\Vcal}(H)=0$ and hence
$u_{\Vcal}(H\cap Z_n)=0$ by monotonicity; and $\{k:d_k\in Z_n\}=\{k:k>n\}$ is
cofinite, so $u_{\Vcal}(Z_n)=1$.  Since $[z\restrict n]=Z_n$,
\[
  P'_n(H)(z)
  =\frac{(1-\alpha)\mu_0(H\cap Z_n)}{\alpha+(1-\alpha)\mu_0(Z_n)}
  \;\le\;\frac{(1-\alpha)\mu_0(Z_n)}{\alpha+(1-\alpha)\mu_0(Z_n)}
  \;\le\;1-\alpha\;<\;\delta ,
\]
the second inequality because $t\mapsto(1-\alpha)t/(\alpha+(1-\alpha)t)$ is
nondecreasing on $[0,1]$ with value $1-\alpha$ at $t=1$.  So
$z\notin C_n(H,\delta)$ for every $n$, that is $z\notin O$.

\smallskip
\noindent\emph{Step 2: an unatomized distinguished path crosses at time $k$.}
Let $k\in J_0$ and $n\ge k$.  Then $[d_k\restrict n]\subseteq S_k$, and since
$z\notin S_k$ while $d_j\in S_j$ with the shells pairwise disjoint, we have
$[d_k\restrict n]\cap(D\cup\{z\})=\{d_k\}$ and therefore
$H\cap[d_k\restrict n]=[d_k\restrict n]\setminus\{d_k\}$.  Also
$\{j:d_j\in[d_k\restrict n]\}=\{k\}$ is finite, so
$u_{\Vcal}([d_k\restrict n])=0$ and hence
$P'([d_k\restrict n])=(1-\alpha)\mu_0([d_k\restrict n])$, which is positive by
\eqref{eq:posi}.  Consequently
\[
  P'_n(H)(d_k)
  =\frac{\mu_0\bigl([d_k\restrict n]\setminus\{d_k\}\bigr)}
        {\mu_0\bigl([d_k\restrict n]\bigr)}
  =1-\frac{\mu_0(\{d_k\})}{\mu_0\bigl([d_k\restrict n]\bigr)}
  =1\;>\;\delta ,
\]
because $\mu_0(\{d_k\})=0$.  In particular $d_k\in C_k(H,\delta)\subseteq O$.

\smallskip
\noindent\emph{Step 3: conclusion.}
By Step~2, $\{k:d_k\in O\}\supseteq J_0$, which lies in $\Vcal$; so
$u_{\Vcal}(O)=1$ and
$P'(O)=\alpha+(1-\alpha)\mu_0(O)$.  Each $O_N$ is clopen, so
$P'(O_N)=R'(O_N)$ by \cref{lem:companion}, where
$R'=\alpha\delta_z+(1-\alpha)\mu_0$; the events $O_N$ increase to $O$ and $R'$
is countably additive, so $\lim_NP'(O_N)=R'(O)$, and $R'(O)=(1-\alpha)\mu_0(O)$
because $z\notin O$ by Step~1.  Subtracting gives the displayed difference
$\alpha>0$, and CHT is equivalent to that difference vanishing by
\eqref{eq:cht-continuity}.

\smallskip
For $k\in J_0$ the crossing time is exactly $k$: for $n<k$ one has
$[d_k\restrict n]=Z_n=[z\restrict n]$, so Step~1 gives
$P'_n(H)(d_k)\le1-\alpha<\delta$.  The crossing time therefore moves with the
index, which is the mechanism \cref{rem:threshold-diagonal} describes and the
one the atom of \eqref{eq:R0} prevents.
\end{proof}

One family covered by \cref{prop:noatom} is the one instantiated by the example
\citet[proof of Theorem~5, p.~411]{nielsen2021convergence} gives.

\begin{corollary}\label{cor:nielsen-thm5}
Let $Q$ be the fair-coin measure \eqref{eq:faircoin}, let $\Vcal$ be a free
ultrafilter on $\N$, let $u_{\Vcal}$ be defined from $\Vcal$ by \eqref{eq:u},
and for $\alpha\in(0,1)$ put $P_\alpha=\alpha u_{\Vcal}+(1-\alpha)Q$.  Then
$P_\alpha$ satisfies \eqref{eq:posi} and does not have the CHT property, and
distinct values of $\alpha$ give distinct probability functions.  In particular
there are uncountably many such $P_\alpha$.
\qedthm \end{corollary}

\begin{proof}
Positivity: $P_\alpha([\omega\restrict n])\ge(1-\alpha)Q([\omega\restrict n])
=(1-\alpha)2^{-n}>0$ by \eqref{eq:faircoin}.  The measure $Q$ is diffuse, so
$Q(\{z\})=0$ and $Q(\{d_k\})=0$ for every $k$; hence $J_0=\N\in\Vcal$ in
\eqref{eq:noatoms} and \cref{prop:noatom} applies.  For distinctness, $D$ is
countable, so $Q(D)=0$ by diffuseness, while $u_{\Vcal}(D)=1$; hence
$P_\alpha(D)=\alpha$.
\end{proof}

\begin{remark}[The status of Nielsen's Theorem~5]\label{rem:nielsen-thm5}
Theorem~5 \citep[p.~411]{nielsen2021convergence} makes two claims: that there
are uncountably many probability functions satisfying \eqref{eq:posi} with the
approximation property that do not converge to the truth almost surely, and
that they therefore do not have the CHT property.  The second is inferred from
the first ``by Theorem~2 and Lemma~3''
\citep[p.~411]{nielsen2021convergence} --- that is, from the sufficiency
direction of \eqref{eq:characterization}, since it is only that direction which
yields, for a probability function with the approximation property, that CHT
implies almost-sure convergence.  \cref{thm:counterexample} removes that
direction, so the published derivation of the second claim lapses.
\cref{cor:nielsen-thm5} restores it for the probability functions his own
example yields, which suffices for a claim of existential form, and does so
without appeal to Theorem~2; his first claim is his own and we do not revisit
it.  The same inference is what supports his statement
that the approximation property does not imply the CHT property, one direction
of the independence of the two criteria; that direction is therefore restored as
well, while the other rests on Theorem~4, whose proof
\cref{sec:published-proof} shows to be unavailable.

The construction is his.  His proof takes $D$ to be any infinite discrete
subspace of $\Omega$ with $Q(D\cup L_D)=0$, where $L_D$ is the set of limit
points of $D$, and offers as its example the set of paths $\omega_n$ with
$\omega_n(i)=\ind_{\{n\}}(i)$ for $i\in\N$ --- the distinguished paths $d_n$ of
\eqref{eq:shellsandpaths} --- whose only limit point is the constantly zero
path, our $z$.  For that $D$ he takes a nonprincipal ultrafilter $\Wcal$ on
$\Borel$ with
$D\in\Wcal$, meaning that $\Wcal$ is a filter on $\Borel$ containing exactly one
of $A$ and $\Omega\setminus A$ for each $A\in\Borel$ and containing no
singleton; the two-valued probability $F\mapsto\ind_{\Wcal}(F)$; and the mixture
$\alpha\,\ind_{\Wcal}+(1-\alpha)Q$.  Given such a $\Wcal$, put:
\[
  \Vcal\quad=\quad\bigl\{\,S\subseteq\N:\{d_k:k\in S\}\in\Wcal\,\bigr\}.
\]
Then $\Vcal$ is a free ultrafilter on $\N$, and
$\ind_{\Wcal}=u_{\Vcal}$ on $\Borel$, so his mixture is the $P_\alpha$ of
\cref{cor:nielsen-thm5}.

For the first assertion, the map $S\mapsto\{d_k:k\in S\}$ carries subsets of
$\N$ to subsets of $D$ preserving inclusions and intersections, so $\Vcal$ is
closed upward and under intersections; $\N\in\Vcal$ because $D\in\Wcal$, and
$\varnothing\notin\Vcal$ because $\varnothing\notin\Wcal$.  For each $S$ the
sets $\{d_k:k\in S\}$ and $\{d_k:k\notin S\}$ are disjoint with union $D$, and
$D\in\Wcal$, so exactly one of them lies in $\Wcal$: whichever of
$\{d_k:k\in S\}$ and its complement in $\Omega$ lies in $\Wcal$, intersecting
with $D$ leaves a member of $\Wcal$; and not both of them, since a filter
contains no two disjoint sets.  Hence exactly one of $S$ and
$\N\setminus S$ lies in $\Vcal$.  And if a finite $S$ had
$\{d_k:k\in S\}\in\Wcal$, then splitting that finite set repeatedly and keeping
at each step the half that lies in $\Wcal$ would put a singleton in $\Wcal$,
contradicting nonprincipality; so $\Vcal$ contains no finite set.  For the
second assertion, $A\in\Wcal$ if and only if $A\cap D\in\Wcal$, since $\Wcal$
is closed upward and under intersection with $D\in\Wcal$; and
$A\cap D=\{d_k:k\in\{k:d_k\in A\}\}$, so
$\ind_{\Wcal}(A)=1$ exactly when $\{k:d_k\in A\}\in\Vcal$, which by
\eqref{eq:u} is exactly when $u_{\Vcal}(A)=1$.
\qedthm\end{remark}

\subsection{Failure of Almost-Uniform Convergence}

The event witnessing failure of almost-uniform convergence has a simple
interpretation.  It removes the limit point $z$ and every distinguished point
$d_k$, but retains the diffuse part of each shell.  At $d_k$, its time-$k$
posterior is then exactly the diffuse fraction $\eta_k$.

\begin{proposition}\label{prop:not-au}
There is a Borel event $H$ for which $P_n(H)$ does not converge to $\ind_H$ almost
uniformly.
\qedthm \end{proposition}

\begin{proof}
Let $H$ be the event \eqref{eq:H}, so that $\ind_H(d_k)=0$.  At time $n=k$, the
information cylinder at $d_k$ is the shell $S_k=[0^{k-1}1]$ of
\eqref{eq:shellsandpaths}.  Since $u(S_k)=0$, $\nu_k(H)=1$, and
$R_0(S_k)=t_k$, we have
\begin{equation}\label{eq:error-eta}
  P_k(H)(d_k)=P(H\given S_k)=\eta_k,
\end{equation}
where the subscript $k$ on $P_k$ denotes observation time.

Suppose that $P_n(H)\to\ind_H$ almost uniformly.  In
Definition~\ref{def:convergence}, take the tolerance for the exceptional event to be
$\alpha/2$.  There is then an event $E$ with $P(E)<\alpha/2$ such that
convergence is uniform on $E^c$.  This inequality forces $u(E)=0$: if $u(E)=1$,
then $P(E)\ge\alpha$.  Hence
\begin{equation}\label{eq:Klarge}
  K=\set{k:d_k\notin E}\in\Ucal.
\end{equation}
Uniform convergence on $E^c$, evaluated at $\omega=d_k$ and $n=k$, implies from
\eqref{eq:error-eta} that $\eta_k\to0$ along $K$.  This contradicts
Lemma~\ref{lem:diagonal-ultrafilter}(b).
\end{proof}

\subsection{Almost-Sure Convergence}

Almost-sure convergence survives, and not by accident.  \cref{prop:noatom}
showed that CHT forces the countably additive component to place an atom on the
distinguished paths, on a large set of indices.  Those atoms have a price: they
protect the distinguished paths from the exceptional event of L\'evy's theorem,
and that is enough to deliver almost-sure convergence.  Within the architecture
of \cref{sec:construction}, therefore, CHT implies almost-sure convergence to
the truth --- a statement of the shape of Nielsen's Corollary~1, and one that
does not use the approximation property.

\begin{proposition}[CHT forces almost-sure convergence]\label{prop:cht-as}
Let $\Vcal$ be a free ultrafilter on $\N$, let $u_{\Vcal}$ be defined from
$\Vcal$ by \eqref{eq:u}, let $\mu_0$ be a countably additive Borel probability
on $\Omega$ with $\mu_0(\{z\})=0$, let $\alpha\in(0,1)$, and suppose that
$P'=\alpha u_{\Vcal}+(1-\alpha)\mu_0$ satisfies \eqref{eq:posi}.  If $P'$ has
the CHT property, then the posteriors of $P'$ converge to the truth almost
surely for every Borel event.
\qedthm \end{proposition}

\begin{proof}
By \cref{prop:noatom}, the set
$J=\{k\in\N:\mu_0(\{d_k\})>0\}$ belongs to $\Vcal$: otherwise its complement
$J_0$ would be $\Vcal$-large by the ultrafilter property, and CHT would fail.

Fix $A\in\Borel$, let $A^{\sharp}$ be as in \eqref{eq:Asharp}, and put
$R'=\alpha\delta_z+(1-\alpha)\mu_0$ as in \eqref{eq:companion-pair}.  The
measure $R'$ gives every nonempty cylinder positive probability: if $z\in B$
then $R'(B)\ge\alpha$, while if $z\notin B$ then $B\subseteq S_k$ for a unique
$k$, so $u_{\Vcal}(B)=0$ --- no $d_j$ other than $d_k$ meets $B$, and a free
ultrafilter contains no finite set --- and \eqref{eq:posi} gives
$0<P'(B)=(1-\alpha)\mu_0(B)\le R'(B)$.  So
\cref{lem:standard-facts}(a) applies to $R'$, and with
\eqref{eq:posterior-companion} it gives
\[
  P'_n(A)\;=\;R'\bigl(A^{\sharp}\given[\,\cdot\restrict n]\bigr)
  \;\longrightarrow\;\ind_{A^{\sharp}}
  \qquad R'\text{-almost surely}.
\]
Let $N_A\in\Borel$ be an $R'$-null event off which this convergence holds.  An
$R'$-null event contains no point of positive $R'$-measure, and
$R'(\{z\})=\alpha>0$ while $R'(\{d_k\})=(1-\alpha)\mu_0(\{d_k\})>0$ for
$k\in J$; hence $z\notin N_A$ and $d_k\notin N_A$ for every $k\in J$.
Therefore $\{k:d_k\in N_A\}\subseteq\N\setminus J$, which is not $\Vcal$-large,
so $u_{\Vcal}(N_A)=0$; and $(1-\alpha)\mu_0\le R'$ gives $\mu_0(N_A)=0$.  Hence
$P'(N_A)=0$.

Off $N_A$ the posteriors converge to $\ind_{A^{\sharp}}$.  Since
$A\symdiff A^{\sharp}\subseteq\{z\}$ and
$P'(\{z\})=\alpha u_{\Vcal}(\{z\})+(1-\alpha)\mu_0(\{z\})=0$ --- the first term
vanishing because no $d_k$ equals $z$, so that
$\{k:d_k\in\{z\}\}=\varnothing\notin\Vcal$, and the second by hypothesis ---
the posteriors converge to $\ind_A$ off the event
$N_A\cup\{z\}$, which has $P'$-probability zero by finite additivity.
\end{proof}

\begin{corollary}\label{prop:as}
Despite \cref{prop:not-au}, the posteriors of $P$ converge to the truth almost
surely for every Borel event.
\qedthm \end{corollary}

\begin{proof}
Apply \cref{prop:cht-as} with $\Vcal=\Ucal$ and $\mu_0=R_0$, which is
legitimate because $R_0(\{z\})=0$, because $P$ satisfies \eqref{eq:posi} by
\cref{prop:basic-properties}, and because $P$ has the CHT property by
\cref{prop:cht}.  Alternatively, apply the proof directly: $R_0$ has an atom at
every $d_k$ by \eqref{eq:R0}, so $J=\N$.
\end{proof}

\subsection{Refutation of the Characterization}

\begin{theorem}[Counterexample to Nielsen's Theorem~2]\label{thm:counterexample}
There exists a merely finitely additive probability function on the Borel subsets of
$\{0,1\}^{\N}$ that satisfies \eqref{eq:posi}, has the approximation property, and
has the CHT property, but whose posteriors do not converge to the truth almost
uniformly.  The probability can be chosen so that its posteriors converge to the
truth almost surely.
\qedthm \end{theorem}

\begin{proof}
Combine \cref{prop:basic-properties,prop:approximation,prop:cht,prop:not-au}
and \cref{prop:as}.
\end{proof}

\begin{remark}[Where diagonalization fails, concretely]\label{rem:threshold-diagonal}
\cref{lem:diagonal-ultrafilter}(a) is exactly what the CHT verification uses:
for each fixed $\delta$, the set of indices $k$ with $\eta_k\le\delta$ is
$\Ucal$-large, and that suffices to place the infinite hitting union on the
correct side of the ultrafilter.  Almost-uniform convergence would require one
$\Ucal$-large set of indices on which the corresponding error bounds hold
eventually for every $\delta$ at once, and
\cref{lem:diagonal-ultrafilter}(b) rules that out.  This is the separation
between \eqref{eq:weak-quantifiers} and \eqref{eq:au-quantifiers} in concrete
form: the exceptional event may be chosen after the accuracy threshold but not
before it.
\qedthm \end{remark}

\begin{remark}[The status of Nielsen's Corollary~1]\label{rem:corollary}
\cref{prop:as} shows that the posteriors of $P$ do converge to the truth almost
surely.  \cref{thm:counterexample} therefore does not refute Nielsen's
Corollary~1 \citep[p.~408]{nielsen2021convergence}, which draws the almost-sure
conclusion from approximation and CHT.  What it does remove is that corollary's
only published proof, which applies Theorem~2 and then passes from
almost-uniform to almost-sure convergence.  Whether approximation and CHT
together imply almost-sure convergence to the truth for an arbitrary
probability function is, so far as we know, open.  What
\cref{prop:cht-as} settles is that no counterexample can be found in the family
to which the construction of \cref{sec:construction} belongs: for a mixture of
a two-valued ultrafilter probability read along the distinguished paths with a
countably additive probability giving $z$ probability zero, subject to
\eqref{eq:posi}, CHT alone --- without the approximation property --- already
implies almost-sure convergence to the truth.  The almost-sure convergence of
$P$ is thus forced by its having CHT, not an artifact of the particular weights
chosen.
\qedthm \end{remark}

\section{Existence Without Extension}\label{sec:existence}

Having established that one of the principal results stated by
\citet{nielsen2021convergence} is false, we take up the proof he gives for a
second.  That proof fails for an unrelated reason, which this section
identifies: it requires nontrivial finitely additive extensions of a Borel
probability agreeing with it on every open set, and regularity leaves no room
for any.  The statement the proof was meant to establish is nevertheless true,
and the construction of \cref{sec:construction} proves it once the sequence
$(\eta_k)$ is replaced by a convergent one.

\subsection{No Nontrivial Extensions Agreeing on Open Sets}\label{sec:published-proof}

The proof of \citet[Theorem~3, pp.~409--410]{nielsen2021convergence} begins
with the fair-coin Borel probability $Q$ and the
algebra $\Acal$ generated by all open subsets of $\Omega$.  It claims that there are
nontrivial finitely additive Borel probabilities agreeing with $Q$ on $\Acal$.
Regularity rules this out.

\begin{proposition}[Uniqueness from agreement on open sets]
\label{prop:open-uniqueness}
Let $\mu$ be a countably additive Borel probability on $\Omega$.  If a
probability function $P$ on $\Borel$ agrees with $\mu$ on every open set, then
$P=\mu$ on every Borel set.\qedthm
\end{proposition}

\begin{proof}
Agreement on open sets implies agreement on closed sets by complementation.  Let
$B\in\Borel$ and $\eps>0$.  By \cref{lem:regularity}, choose a closed $F$ and an
open $G$ such that:
\[
  F\subseteq B\subseteq G,
  \qquad \mu(G\setminus F)<\eps.
\]
Monotonicity and the assumed agreement give
\[
  \mu(F)=P(F)\le P(B)\le P(G)=\mu(G).
\]
The interval $[\mu(F),\mu(G)]$ has length less than $\eps$ and contains
$\mu(B)$ as well, so
$\abs{P(B)-\mu(B)}<\eps$.  Letting $\eps\downarrow0$ yields $P(B)=\mu(B)$.
\end{proof}

\begin{corollary}\label{cor:singleton-extension}
If $Q$ is the fair-coin measure and $\Acal$ is the algebra generated by the open
sets, then the set of finitely additive Borel extensions of $Q\restrict\Acal$ is the
singleton $\{Q\}$.
\qedthm \end{corollary}

\begin{proof}
Every open set belongs to $\Acal$, so a finitely additive Borel extension of
$Q\restrict\Acal$ agrees with the countably additive probability $Q$ on every
open set.  \cref{prop:open-uniqueness} gives equality on every Borel set.
\end{proof}

The failure here is one of nonuniqueness, not of existence.  Extensions of a
charge from a subalgebra to a larger algebra always exist
\citep[Corollary~3.3.4, p.~74]{rao1983charges}, and the algebra $\Acal$ is a
proper subalgebra of $\Borel$; what
\cref{cor:singleton-extension} says is that on $\Acal$ the extension is already
determined, because the open sets pin down a countably additive Borel
probability among all the finitely additive ones.  The published proof needs
extensions that differ from $Q$, and there are none.

The specific extension formula in the published proof is also not well defined.
Let $D$ be a countable dense subset of $\Omega$.  The proof proposes:
\begin{equation}\label{eq:bad-extension}
  \widetilde P\bigl((A_1\cap D)\cup(A_2\cap D^c)\bigr)=Q(A_1),
  \qquad A_1,A_2\in\Acal.
\end{equation}
Because $Q(D)=0$, regularity gives an open set $G\supseteq D$ with
$Q(G)<1/2$.\footnote{The appeal to regularity is convenient but not necessary:
because $D$ is
countable, an open set of small measure containing it can be exhibited
outright. Enumerate $D=\set{x_1,x_2,\dots}$ and, for each $j\in\N$, put
\[
  B_j=\bigl[\,x_j\restrict(j+2)\,\bigr],
\]
the cylinder of length $j+2$ containing $x_j$. Since $Q$ is the fair-coin
measure, $Q\bigl(\bigl[\,\omega\restrict n\,\bigr]\bigr)=2^{-n}$ for every
$\omega\in\Omega$ and every $n\in\N$, so $Q(B_j)=2^{-j-2}$. The set
$G=\bigcup_{j\ge 1}B_j$ is a union of cylinders and hence open; it contains
$D$, since $x_j\in B_j$ for each $j$; and countable subadditivity of $Q$
gives
\[
  Q(G)\;\le\;\sum_{j=1}^{\infty}Q(B_j)
  \;=\;\sum_{j=1}^{\infty}2^{-j-2}
  \;=\;\tfrac14\;<\;\tfrac12 .
\]
This construction uses only countable additivity of $Q$ and the cylinder
values $Q\bigl(\bigl[\,\omega\restrict n\,\bigr]\bigr)=2^{-n}$, and it
dispenses with \cref{lem:regularity} at this point. It also makes the
underlying phenomenon visible: the cylinder around each point of $D$ may be
taken as short as one pleases, so a dense open set can carry arbitrarily
small probability. Density is a topological property and imposes no lower
bound on measure. The set $D$ itself need not belong to $\Acal$; what the
argument requires is that $G$ belong to $\Acal$, which holds because $G$ is
open.}  But:
\[
  D=(\Omega\cap D)\cup(\varnothing\cap D^c)
   =(G\cap D)\cup(\varnothing\cap D^c),
\]
while \eqref{eq:bad-extension} assigns the two representations the respective
values $1$ and $Q(G)<1/2$.

\cref{cor:singleton-extension} therefore invalidates the published proof of
Nielsen's Theorem~3, and with it the proof of Theorem~4, which draws on two
distinct extreme points of the same extension set.  \cref{thm:positive} below
supplies a replacement proof for the statement of Theorem~3; we leave the
standalone truth of Theorem~4 open.  Nielsen's Theorem~5
\citep[p.~411]{nielsen2021convergence} rests on a separate construction, not on
the extension set, and is unaffected by the argument of this subsection.  Its
second claim is nevertheless left without a proof by the refutation of
\cref{sec:main}; \cref{rem:nielsen-thm5} records what becomes of it.

\subsection{An Almost-Uniformly Convergent Family}\label{sec:positive}

The construction of \cref{sec:construction} yields an almost-uniformly convergent
probability function once the sequence $(\eta_k)$ is replaced by an ordinarily
convergent one.  This proves the statement of
\citet[Theorem~3, p.~409]{nielsen2021convergence}, and with it the assertion
\citet[p.~410]{nielsen2021convergence} draws from that theorem, that uncountably
many merely finitely additive probability functions converge to the truth
almost surely.

Let $\Ucal$ now be any free ultrafilter on $\N$, and let
$(\theta_k)\subset(0,1)$ satisfy
\begin{equation}\label{eq:theta-zero}
  \theta_k\longrightarrow0.
\end{equation}
The point $z$, the shells $S_k$ and distinguished paths $d_k$, the weights
$t_k=2^{-k}$, and the measures $\nu_k$ are those of \cref{sec:construction};
only the diffuse fractions and the ultrafilter change.  Define $R_0^{\theta}$,
$P_\alpha^{\theta}$, and $R_\alpha^{\theta}$ by replacing $\eta_k$ with $\theta_k$
in \eqref{eq:R0}, \eqref{eq:P}, and \eqref{eq:R-companion}, respectively, and let
$u$ be given by \eqref{eq:u} for this $\Ucal$.

\begin{theorem}[Replacement for Nielsen's Theorem~3]\label{thm:positive}
For every $\alpha\in(0,1)$, $P_\alpha^{\theta}$ is merely finitely additive,
satisfies \eqref{eq:posi}, and converges to the truth almost uniformly.  Distinct
values of $\alpha$ give distinct probability functions.  Consequently there are
uncountably many merely finitely additive probability functions satisfying
\eqref{eq:posi} and converging to the truth almost uniformly.
\qedthm \end{theorem}

\begin{proof}
Write $D=\{d_k:k\in\N\}$ as in \eqref{eq:Dset}.
Finite additivity, failure of countable additivity, positivity, and clopen
approximation are proved exactly as in \cref{prop:basic-properties}
and \cref{prop:approximation}.  We prove almost-uniform convergence for an
arbitrary $A\in\Borel$ and $\eps>0$.

\smallskip
\noindent\emph{Step 1: a countably additive version of the posterior process.}
Put $a=u(A)$ and define $A^{\sharp}$ by \eqref{eq:Asharp}.
\cref{lem:companion}, applied with $\Vcal=\Ucal$ and $\mu_0=R_0^{\theta}$,
gives
\begin{equation}\label{eq:theta-companion}
  P_{\alpha,n}^{\theta}(A)(\omega)
  =R_\alpha^{\theta}
     \bigl(A^{\sharp}\given[\omega\restrict n]\bigr),
\end{equation}
where $P_{\alpha,n}^{\theta}(A)$ denotes the time-$n$ posterior generated by
$P_\alpha^\theta$.  By Lemma~\ref{lem:standard-facts}(a), the right-hand side
converges to $\ind_{A^{\sharp}}$ $R_\alpha^\theta$-almost surely and hence
$R_0^\theta$-almost surely.  Lemma~\ref{lem:standard-facts}(b) supplies a Borel
event $E_0$ such that
\begin{equation}\label{eq:E0}
  R_0^{\theta}(E_0)<\frac{\eps}{4(1-\alpha)}
\end{equation}
and the convergence in \eqref{eq:theta-companion} is uniform on $E_0^c$.

\smallskip
\noindent\emph{Step 2: choose one exceptional event with small $P_\alpha^\theta$-mass.}
Let
\[
  M=\{k:\ind_A(d_k)\ne a\}.
\]
The set $M$ is not in $\Ucal$: it is the complement of
$\{k:d_k\in A\}$ when $a=1$, and equals that set when $a=0$.  Choose $K_0$
so large that
\begin{equation}\label{eq:tail-small}
  \sum_{k>K_0}t_k<\frac{\eps}{4(1-\alpha)}
\end{equation}
and set
\begin{equation}\label{eq:E-positive}
  E=(E_0\setminus D)\cup\{z\}
    \cup\{d_k:k>K_0,\ k\in M\}.
\end{equation}
The first two pieces contain no distinguished point $d_k$, and the index set of
the third piece is contained in $M\notin\Ucal$.  Hence $u(E)=0$.  Moreover,
\[
 R_0^\theta(E)
 \le R_0^\theta(E_0)+\sum_{k>K_0}t_k
 <\frac{\eps}{2(1-\alpha)}.
\]
Therefore
\[
  P_\alpha^\theta(E)=(1-\alpha)R_0^\theta(E)<\frac{\eps}{2}<\eps.
\]

\smallskip
\noindent\emph{Step 3: uniform convergence away from the distinguished points.}
If $\omega\in E^c\setminus(D\cup\{z\})$, then $\omega\notin E_0$ and
$\ind_{A^{\sharp}}(\omega)=\ind_A(\omega)$.  Uniform convergence on these points
therefore follows from the definition of $E_0$.

\smallskip
\noindent\emph{Step 4: uniform convergence on the remaining $d_k$.}
Let
\[
  q_n=P_\alpha^{\theta}(A\given Z_n).
\]
Exactly as in the proof of \cref{prop:cht}, $q_n\to a$.  For $n\ge k$,
\eqref{eq:shell-estimate} applied to $A$ if $d_k\notin A$, and to $A^c$ if
$d_k\in A$, gives
\begin{equation}\label{eq:theta-shell-bound}
  \left|P_{\alpha,n}^{\theta}(A)(d_k)-\ind_A(d_k)\right|\le\theta_k.
\end{equation}
For each fixed $k$, the cylinders $[d_k\restrict n]$ decrease to $\{d_k\}$.
Continuity from above of $R_\alpha^\theta$ therefore gives
\begin{equation}\label{eq:fixed-dk}
  P_{\alpha,n}^{\theta}(A)(d_k)\longrightarrow\ind_A(d_k),
\end{equation}
because $R_\alpha^\theta(\{d_k\})>0$.

Fix an accuracy $r>0$.  Choose $K_1>K_0$ so that $\theta_k<r$ for all
$k\ge K_1$, and choose $N_0$ so that $|q_n-a|<r$ for $n\ge N_0$.  By \eqref{eq:fixed-dk}, there is an integer
$N_1$ such that
$\bigl|P_{\alpha,n}^{\theta}(A)(d_k)-\ind_A(d_k)\bigr|<r$ for every $n\ge N_1$
and every $k<K_1$.  Let $N=\max\{N_0,N_1\}$.  If $d_k\in E^c$ and $n\ge N$, then one of
three cases applies:
\begin{enumerate}[(i)]
\item If $k<K_1$, the error is less than $r$ by the choice of $N_1$.
\item If $k\ge K_1$ and $n\ge k$, the error is less than $r$ by
      \eqref{eq:theta-shell-bound}.
\item If $k\ge K_1$ and $n<k$, then $k>K_0$ and $d_k\notin E$ imply
      $\ind_A(d_k)=a$.  Since $[d_k\restrict n]=Z_n$, the error is
      $|q_n-a|<r$.
\end{enumerate}
Thus the convergence is uniform on the distinguished points remaining in $E^c$
and, with Step~3, on all of $E^c$.

\smallskip
\noindent\emph{Step 5: uncountably many distinct probabilities.}
Let
\[
  c=\sum_{k=1}^{\infty}t_k(1-\theta_k).
\]
Since every $\theta_k$ is positive, $c<1$.  As $u(D)=1$,
\[
  P_\alpha^\theta(D)=c+\alpha(1-c),
\]
which is strictly increasing in $\alpha$.  Hence distinct values of $\alpha$ give
distinct probability functions, and the family is uncountable.
\end{proof}

\section{Almost-Sure Learning Does Not Characterize Countable Additivity}\label{sec:as-ca}

Everything so far has turned on almost-uniform convergence.
\citet[p.~408]{nielsen2021convergence} leaves the characterization of the
weaker notion, almost-sure convergence to the truth, as an open problem.  This
section closes off its most immediate candidate answer, countable additivity
itself, and then shows something stronger: the merely finitely additive
component of the witnessing probability function is invisible not merely to
almost-sure learning but to every finite-data conditional probability.

\subsection{The Conjecture}\label{sec:conjecture}

\cref{thm:counterexample} separates almost-uniform convergence from the
conjunction of approximation and CHT.  A second and more elementary question
concerns almost-sure convergence on its own, and the most immediate candidate
for a characterizing condition is countable additivity itself.  One direction is
classical: a countably additive probability satisfying \eqref{eq:posi} converges
to the truth almost surely, which is \cref{lem:standard-facts}(a) applied with
$\mu=P$.  It is natural to ask whether the converse also holds.

\begin{conj}[A natural but false conjecture]
\label{conj:as-iff-ca}
Let $P$ be a probability function on $\Borel$ satisfying \eqref{eq:posi}.  Then
$P$ is countably additive if and only if:
\[
    P_n(A)\longrightarrow \ind_A
    \qquad P\text{-almost surely}
\]
for every $A\in\Borel$.
\qedthm \end{conj}

Recall from \cref{def:convergence}(a) that the exceptional null event in this
statement is allowed to depend on $A$; \cref{sec:finite-data} returns to that
point.  This section refutes the conjecture by an argument independent of
\cref{sec:construction,sec:main}.

\subsection{A Merely Finitely Additive Counterexample}

\cref{thm:positive} already refutes \cref{conj:as-iff-ca}: the probability
functions $P_\alpha^{\theta}$ constructed there are merely finitely additive and
converge to the truth almost uniformly, hence almost surely.  The construction
below is more elementary, and it yields more than the failure of the conjecture.
A merely finitely additive component can be invisible on every
finite-information event --- not merely small there, but matched exactly by a
countably additive probability --- while still destroying countable additivity
on infinitary events.  The construction repeats, in simpler form, the pattern of
\cref{sec:construction}, and repeats the few facts from it that the argument
needs.

\begin{theorem}
\label{thm:as-learning-not-ca}
There exists a merely finitely additive probability function
$P:\Borel\to[0,1]$ satisfying \eqref{eq:posi} such that, for every
$A\in\Borel$,
\[
    P_n(A)\longrightarrow \ind_A
    \qquad P\text{-almost surely}.
\]
Consequently, Conjecture~\ref{conj:as-iff-ca} is false.
\qedthm \end{theorem}

\begin{proof}
Let
\[
    z=(0,0,0,\ldots)\in\Omega
\]
and, for $k\in\N$, let
\[
    d_k
    =
    0^{k-1}10^{\infty},
\]
so that $d_k$ has a $1$ in coordinate $k$ and a $0$ in every other coordinate.
Thus $d_k\to z$ in the Cantor topology.

Let $Q$ again be the fair-coin measure of \cref{sec:geometry}, and define the
countably additive probability
\begin{equation}
\label{eq:app-mu}
    \mu
    :=
    \frac{1}{2} Q
    +
    \frac{1}{2}\sum_{k=1}^{\infty}2^{-k}\delta_{d_k}.
\end{equation}
Then
\[
    \mu(\{z\})=0,
    \qquad
    \mu(\{d_k\})=2^{-k-1}>0
\]
for every $k$.

Let $\Ucal$ be a free ultrafilter on $\N$. Define
$u:\Borel\to\{0,1\}$ by
\begin{equation}
\label{eq:app-u}
    u(A)
    :=
    \begin{cases}
       1,
       &
       \text{if }\{k\in\N:d_k\in A\}\in\Ucal,
       \\[1mm]
       0,
       &
       \text{if }\{k\in\N:d_k\in A\}\notin\Ucal.
    \end{cases}
\end{equation}
Equivalently, $u(A)$ is the $\Ucal$-limit \eqref{eq:ulimit} of the sequence of
truth values of $A$ along the distinguished paths:
\[
    u(A)=\lim_{k\to\Ucal}\ind_A(d_k).
\]
By \cref{lem:u-properties}, whose proof uses nothing about $\Ucal$ beyond its
being an ultrafilter, $u$ is a two-valued probability function on $\Borel$.

Now put
\begin{equation}
\label{eq:app-P}
    P
    :=
    \frac{1}{2}u+\frac{1}{2}\mu.
\end{equation}

\medskip
\noindent
\emph{Step 1: $P$ satisfies \eqref{eq:posi}.}
Since $\mu\ge \frac{1}{2} Q$, every nonempty cylinder $C$ satisfies
\[
    P(C)
    \ge
    \frac{1}{2}\mu(C)
    \ge
    \frac{1}{4} Q(C)
    >0.
\]
Hence all conditional probabilities appearing below are well defined.

\medskip
\noindent
\emph{Step 2: $P$ is not countably additive.}
Let
\[
    D:=\{d_k:k\in\N\}.
\]
Since every $d_k$ belongs to $D$,
\[
    u(D)=1,
\]
whereas $Q(D)=0$ and hence
\[
    \mu(D)=\frac{1}{2}.
\]
Therefore
\begin{equation}
\label{eq:app-PD}
    P(D)
    =
    \frac{1}{2}+\frac{1}{4}
    =
    \frac{3}{4}.
\end{equation}

On the other hand, a free ultrafilter contains no singleton, so
\[
    u(\{d_k\})=0
\]
for every $k$. It follows from
\eqref{eq:app-mu}--\eqref{eq:app-P} that
\[
    P(\{d_k\})
    =
    2^{-k-2}.
\]
Consequently,
\[
    \sum_{k=1}^{\infty}P(\{d_k\})
    =
    \frac{1}{4}
    \ne
    \frac{3}{4}
    =
    P(D).
\]
Thus $P$ is not countably additive.

\medskip
\noindent
\emph{Step 3: on finite-information events, $u$ behaves like point mass at
$z$.}
If $C$ is clopen, then membership in $C$ is determined by finitely many
coordinates. Because $d_k\to z$, there exists $K$ such that
\[
    \ind_C(d_k)=\ind_C(z)
    \qquad(k\ge K).
\]
Since every cofinite subset of $\N$ belongs to the free ultrafilter
$\Ucal$,
\begin{equation}
\label{eq:app-u-clopen}
    u(C)=\ind_C(z)
    \qquad(C\text{ clopen}).
\end{equation}

Define the countably additive probability
\begin{equation}
\label{eq:app-R}
    R
    :=
    \frac{1}{2}\delta_z+\frac{1}{2}\mu.
\end{equation}
Equation~\eqref{eq:app-u-clopen} implies
\begin{equation}
\label{eq:app-P-R-clopen}
    P(C)=R(C)
    \qquad
    \text{for every clopen }C.
\end{equation}

\medskip
\noindent
\emph{Step 4: every $P$-posterior is an ordinary $R$-posterior for a
slightly modified event.}
Fix $A\in\Borel$, and define
\begin{equation}
\label{eq:app-Asharp}
    A^\sharp
    :=
    \begin{cases}
       A\cup\{z\},&u(A)=1,\\
       A\setminus\{z\},&u(A)=0.
    \end{cases}
\end{equation}
Thus
\begin{equation}
\label{eq:app-Asharp-diff}
    A^\sharp\symdiff A\subseteq\{z\}
\end{equation}
and
\begin{equation}
\label{eq:app-Asharp-z}
    \ind_{A^\sharp}(z)=u(A).
\end{equation}

The set function $u$ is multiplicative:
\begin{equation}
\label{eq:app-u-multiplicative}
    u(A\cap B)=u(A)u(B)
    \qquad(A,B\in\Borel).
\end{equation}
Indeed, if either $u(A)$ or $u(B)$ is zero, the corresponding index set does
not belong to $\Ucal$, and neither does its intersection with the other;
if both are one, their intersection belongs to $\Ucal$.

Let $C$ be clopen. By
\eqref{eq:app-u-clopen},
\eqref{eq:app-Asharp-z}, and
\eqref{eq:app-u-multiplicative},
\begin{equation}
\label{eq:app-u-intersection}
    u(A\cap C)
    =
    u(A)\ind_C(z)
    =
    \delta_z(A^\sharp\cap C).
\end{equation}
Moreover, $\mu(\{z\})=0$, so
\eqref{eq:app-Asharp-diff} gives
\begin{equation}
\label{eq:app-mu-intersection}
    \mu(A\cap C)
    =
    \mu(A^\sharp\cap C).
\end{equation}
Combining
\eqref{eq:app-P},
\eqref{eq:app-R},
\eqref{eq:app-u-intersection}, and
\eqref{eq:app-mu-intersection}, we obtain
\begin{equation}
\label{eq:app-joint-identity}
    P(A\cap C)
    =
    R(A^\sharp\cap C).
\end{equation}
Together with \eqref{eq:app-P-R-clopen}, this implies
\begin{equation}
\label{eq:app-conditional-identity}
    P(A\given C)
    =
    R(A^\sharp\given C)
\end{equation}
whenever $C$ is a nonempty cylinder.

In particular,
\begin{equation}
\label{eq:app-posterior-identity}
    P_n(A)(\omega)
    =
    R(A^\sharp\given[\omega\restrict n]).
\end{equation}

\medskip
\noindent
\emph{Step 5: the posteriors converge to the truth $P$-almost surely.}
Since $R$ is countably additive and gives every nonempty cylinder positive
probability, \cref{lem:standard-facts}(a) applied to $R$ in place of the
measure named there, and with $B=A^\sharp$, yields an $R$-null event $N_A$ such
that
\begin{equation}
\label{eq:app-R-convergence}
    R(A^\sharp\given[\omega\restrict n])
    \longrightarrow
    \ind_{A^\sharp}(\omega)
    \qquad
    \text{for every }\omega\notin N_A.
\end{equation}

Because
\[
    R
    =
    \frac{1}{2}\delta_z+\frac{1}{2}\mu,
\]
the equality $R(N_A)=0$ implies
\[
    z\notin N_A
    \qquad\text{and}\qquad
    \mu(N_A)=0.
\]
Every $d_k$ has strictly positive $\mu$-mass, so $\mu(N_A)=0$ also implies
\[
    d_k\notin N_A
    \qquad
    \text{for every }k.
\]
Hence
\[
    \{k:d_k\in N_A\}=\varnothing,
\]
and therefore
\[
    u(N_A)=0.
\]
Thus
\begin{equation}
\label{eq:app-P-null}
    P(N_A)=0.
\end{equation}

Also,
\[
    u(\{z\})=0,
    \qquad
    \mu(\{z\})=0,
\]
so
\begin{equation}
\label{eq:app-z-null}
    P(\{z\})=0.
\end{equation}

By \eqref{eq:app-Asharp-diff},
\[
    \ind_{A^\sharp}
    =
    \ind_A
    \qquad
    \text{on }\Omega\setminus\{z\}.
\]
Combining
\eqref{eq:app-posterior-identity},
\eqref{eq:app-R-convergence},
\eqref{eq:app-P-null}, and
\eqref{eq:app-z-null}, we obtain
\[
    P_n(A)(\omega)
    \longrightarrow
    \ind_A(\omega)
\]
for every
\[
    \omega\notin N_A\cup\{z\},
\]
and, by finite additivity,
\[
    P(N_A\cup\{z\})=0.
\]
Therefore
\[
    P_n(A)\longrightarrow \ind_A
    \qquad P\text{-almost surely}.
\]
Since $A\in\Borel$ was arbitrary, $P$ converges to the truth almost
surely for every event, even though $P$ is not countably additive.
\end{proof}

\subsection{Agreement with a Countably Additive Probability on Finite Data}\label{sec:finite-data}

\cref{thm:as-learning-not-ca} shows that almost-sure Bayesian learning is too
weak to characterize countable additivity.  The point is not merely that
almost-sure convergence is weaker than almost-uniform convergence.  Rather, by
\eqref{eq:app-conditional-identity}, every finite-data conditional probability
under $P$ coincides with the corresponding conditional probability under a
suitably chosen countably additive probability $R$, computed for an event
modified only at the point $z$.  No amount of finite information distinguishes
the two.  The merely finitely additive component becomes visible on infinitary
events such as
\[
    D=\{d_k:k\in\N\},
\]
where countable additivity fails, but when finite-information posteriors are
calculated it is absorbed into the value of a single $R$-atom.

The order of the quantifiers in that conclusion matters as well. What was
proved above is the eventwise statement
\[
    \forall A\in\Borel\;
    \exists N_A\in\Borel:
    \quad
    P(N_A)=0
    \ \text{and}\
    P_n(A)(\omega)\to\ind_A(\omega)
    \text{ for }\omega\notin N_A.
\]
It does \emph{not} assert the stronger existence of one common probability-one
set on which posterior convergence holds simultaneously for every Borel event,
and the two readings of ``convergence to the truth for all events'' come apart
accordingly.

\section{A Diagonal Strengthening of CHT}\label{sec:diagonal}

CHT is a one-threshold-at-a-time condition.  It gives enough continuity to
control the first crossing of any fixed level, but not the control across
levels needed to choose one small exceptional event for all of them at once.
What is at stake is not finite additivity, which \cref{thm:positive} leaves
compatible with almost-uniform convergence, but whether the error bounds can be
made simultaneous on one probability-large set.

A repair of the characterization would therefore strengthen CHT by a condition
controlling countable unions of hitting events across a sequence of thresholds
tending to zero.  Formulating a minimal and behaviorally transparent version of
that diagonal condition remains an open problem.

\section{What Survives and What Remains Open}\label{sec:conclusion}

\citet[p.~181]{juhl1994realism} urged that countable additivity be scrutinized
as an epistemological axiom rather than accepted as a technical convenience, on
the ground that the convergence theorems it underwrites can fail without it.
\citet[p.~398]{nielsen2021convergence} took up that charge and proposed, for
almost-uniform convergence, a characterization in terms of two behaviorally
interpretable criteria.  This paper shows that the characterization fails in
one direction, and repairs part of what its failure removes.

What survives is this.  The necessity
of \eqref{eq:characterization} holds (\cref{prop:necessity}).  There
are uncountably many merely finitely additive probability functions satisfying
\eqref{eq:posi} whose posteriors converge to the truth almost uniformly
(\cref{thm:positive}), so the statement of Nielsen's Theorem~3 is true although
its published proof is not.  The second claim of his Theorem~5 is true although
its published derivation lapses with the sufficiency direction
(\cref{cor:nielsen-thm5}); one direction of the independence of the two
criteria is restored with it.  Within the family to which the counterexample
belongs, CHT alone implies almost-sure convergence to the truth
(\cref{prop:cht-as}), so the claim
of Nielsen's Corollary~1 holds there.  And almost-sure convergence to the truth
for every event does not characterize countable additivity
(\cref{thm:as-learning-not-ca}).

What fails is sufficiency, the right-to-left direction of \eqref{eq:characterization}.
\cref{thm:counterexample} exhibits a merely finitely additive probability
function that satisfies \eqref{eq:posi}, has the approximation property and
CHT, and converges to the truth almost surely, yet whose posteriors do not
converge almost uniformly.  \cref{sec:diagonal} locates the obstruction.

Four questions are left open.  Do the approximation property and CHT together
imply almost-sure convergence to the truth for an arbitrary probability
function?  That is the claim of Nielsen's Corollary~1;
\cref{rem:corollary} explains why it is now unsupported, and
\cref{prop:cht-as} settles it only within the family the counterexample belongs
to.  Is Nielsen's Theorem~4 true?  \cref{sec:published-proof} removes its proof
without deciding its statement, leaving the other direction of the independence
of the two criteria unsettled.  What condition characterizes almost-sure
convergence to the truth?  \citet[p.~412]{nielsen2021convergence} poses the
problem, and \cref{sec:as-ca} eliminates one candidate.  And what minimal diagonal
strengthening of CHT would restore a characterization of almost-uniform
convergence?  \cref{sec:diagonal} says what such a condition must do without
supplying one.
\bigskip
\backmatter

\bibliography{references}

@article{purves_sudderth_1976,
  author = {Purves, Roger A. and Sudderth, William D.},
  title = {Some Finitely Additive Probability},
  journal = {The Annals of Probability},
  volume = {4},
  number = {2},
  pages = {259--276},
  year = {1976}
}

@book{deFinettiArtofGuessing,
    AUTHOR = {{de Finetti}, Bruno},
     TITLE = {Probability, {I}nduction and {S}tatistics: {T}he {A}rt {o}f
              {G}uessing},
    SERIES = {Wiley Series in Probability and Mathematical Statistics},
 PUBLISHER = {John Wiley \& Sons},
   ADDRESS = {London},
      YEAR = {1972},
   MRCLASS = {60A05},
  MRNUMBER = {MR0440638 (55 \#13512)},
MRREVIEWER = {S. Piccard},
}

@book{rao1983charges,
    AUTHOR = {Bhaskara Rao, K. P. S. and Bhaskara Rao, M.},
     TITLE = {Theory of {C}harges: {A} {S}tudy of {F}initely {A}dditive
              {M}easures},
    SERIES = {Pure and Applied Mathematics},
    VOLUME = {109},
 PUBLISHER = {Academic Press Inc. [Harcourt Brace Jovanovich Publishers]},
   ADDRESS = {New York},
      YEAR = {1983},
   MRCLASS = {28A12 (46E27 46G10)},
  MRNUMBER = {751777 (86f:28006)},
}

@book{billingsley1995prob,
    AUTHOR = {Billingsley, Patrick},
     TITLE = {Probability and {M}easure},
    SERIES = {Wiley Series in Probability and Mathematical Statistics},
   EDITION = {Third},
      NOTE = {A Wiley-Interscience Publication},
 PUBLISHER = {John Wiley \& Sons Inc.},
   ADDRESS = {New York},
      YEAR = {1995},
   MRCLASS = {60-01 (28-01)},
  MRNUMBER = {1324786 (95k:60001)},
}

@book{levy1937theorie,
  title = {Th\'eorie {de} {l'addition} {des} {variables} {al\'eatoires}},
  author = {L{\'e}vy, Paul},
  publisher = {Gauthier-Villars},
  address = {Paris},
  year = {1937}
}

@article{belot2013bayesian,
  title={Bayesian orgulity},
  author={Belot, Gordon},
  journal={Philosophy of Science},
  volume={80},
  number={4},
  pages={483--503},
  year={2013},
  publisher={Cambridge University Press}
}

@article{elga2016bayesian,
  title={Bayesian humility},
  author={Elga, Adam},
  journal={Philosophy of Science},
  volume={83},
  number={3},
  pages={305--323},
  year={2016}
}

@article{nielsen2019obligation,
  title={Obligation, Permission, and {B}ayesian Orgulity},
  author={Nielsen, Michael and Stewart, Rush T},
  journal={Ergo},
  volume={6},
  year={2019},
pages = {59-70},
number = {3}
}

@article{nielsen2021convergence,
  title={Convergence to the Truth Without Countable Additivity},
  author={Nielsen, Michael},
  journal={Journal of Philosophical Logic},
  volume={50},
  number={2},
  pages={395--414},
  year={2021},
  doi={10.1007/s10992-020-09569-2}
}

@article{juhl1994realism,
  author = {Juhl, Cory and Kelly, Kevin T.},
  title = {Realism, Convergence, and Additivity},
  journal = {PSA: Proceedings of the Biennial Meeting of the Philosophy of
             Science Association},
  volume = {1994},
  number = {1},
  pages = {181--189},
  year = {1994},
  publisher = {Cambridge University Press},
  note = {Contributed Papers}
}

@book{jech2003,
  author = {Jech, Thomas},
  title = {Set {T}heory},
  edition = {Third millennium},
  series = {Springer Monographs in Mathematics},
  publisher = {Springer},
  address = {Berlin},
  year = {2003},
  note = {Revised and expanded}
}

@article{wimmers1982,
  author = {Wimmers, Edward L.},
  title = {The {S}helah {$P$}-point independence theorem},
  journal = {Israel Journal of Mathematics},
  volume = {43},
  number = {1},
  pages = {28--48},
  year = {1982}
}

\end{document}